\documentclass{amsart}
\usepackage{amsmath,amsfonts,amsthm}
\usepackage{amssymb,latexsym}
\usepackage{graphics}
\usepackage[colorlinks]{hyperref}

\usepackage[all]{xypic}
\usepackage{amssymb}

\theoremstyle{plain}
\theoremstyle{definition}
\newtheorem{theorem}{Theorem}[section]
\newtheorem{lemma}[theorem]{Lemma}
\newtheorem{proposition}[theorem]{Proposition}

\newtheorem{corollary}[theorem]{Corollary}

\newtheorem{exam}[theorem]{Example}
\newtheorem{counterexample}[theorem]{Counterexample}

\newtheorem{note}[theorem]{Note}

\newtheorem{convention}[theorem]{Convention}
\newtheorem{rem}[theorem]{Remark}
\theoremstyle{remark}

\numberwithin{equation}{section}
\newcommand{\SP}{\: \: \: \: \:}

\newcommand{\eset}{{\rm ent}_{\rm set}}
\newcommand{\ecset}{{\rm ent}_{\rm cset}}
\title{Set--theoretical entropies of weighted generalized shifts}
\author[F. Ayatollah Zadeh Shirazi, A. Hosseini, L. Mousavi, R. Rezavand]{Fatemah Ayatollah Zadeh Shirazi,  
Arezoo Hosseini, \\ Lida Mousavi, Reza Rezavand}
\begin{document}
\begin{abstract}
In this paper for a finite field $F$, a nonempty set
$\Gamma$, a self--map $\varphi:\Gamma\to\Gamma$ and a weight vector
$\mathfrak{w}\in F^\Gamma$, we show that the set--theoretical entropy
of the weighted generalized shift
$\sigma_{\varphi,\mathfrak{w}}:F^\Gamma\to F^\Gamma$ is either
zero or $+\infty$, moreover it is equal to zero if and only if
$\sigma_{\varphi,\mathfrak{w}}$ is quasi--periodic. On the other
hand after characterizing all conditions under which
$\sigma_{\varphi,\mathfrak{w}}:F^\Gamma\to F^\Gamma$ is of finite
fibre, we show that the contravariant set--theoretical entropy of a finite
fibre $\sigma_{\varphi,\mathfrak{w}}:F^\Gamma\to F^\Gamma$
depends only on $\varphi$ and
$\rm{supp}(\mathfrak{w})$. In final sections we study the restriction of  $\sigma_{\varphi,\mathfrak{w}}$
to the direct sum $\mathop{\bigoplus}\limits_{\Gamma}F$.
\end{abstract}
\maketitle
\noindent {\small {\bf 2020 Mathematics Subject Classification:}  03E20, 37B99  \\
{\bf Keywords:}}  Contravariant set-- theoretical entropy, Covariant set--theoretical entropy, Infinite anti--orbit number, Infinite orbit number, Weighted generalized shift.
\section{Introduction}\label{sec1}
\noindent As it has been mentioned in several texts ``Entropy''
could be defined as the (numerical) value of dynamicity/
uncertainty/ complexity in a system. Let's name some types of
entropy:  measure entropy~\cite{kolomogrof, sinai, walters},
topological entropy~\cite{adler1, adler}, algebraic
entropy~\cite{algd, weiss}, adjoint entropy~\cite{adjoint}.
\\
Bernoulli shifts and entropy are the common interests of many not only old researches~\cite{orn1, orn2}, but also new ones~\cite{neww}.
According to the preface of~\cite{bern},
Ornstein (and others) have shown
``that a large class of transformations of physical
and mathematical interest are isomorphic to Bernoulli shifts''. 
During the works of Ornstein and other mathematicains, we sense strongly that
Bernoulli shifts are classical and the simplest examples in ergodic and entropy theory (see e.g. the third paragraph of~\cite{orn1}).
Without any doubt, the left Bernoulli shift
$\mathop{\{1,\ldots,k\}^{\mathbb N}\to\{1,\ldots,k\}^{\mathbb N}}\limits_{(x_n)_{n\geq1}\mapsto(x_{n+1})_{n\geq1}}$
and two--sided Bernoulli shift $\mathop{\{1,\ldots,k\}^{\mathbb Z}\to\{1,\ldots,k\}^{\mathbb Z}}\limits_{(x_n)_{n\in\mathbb{Z}}\mapsto(x_{n+1})_{n\in\mathbb{Z}}}$ have great part to inspiring generalized shift.
``Generalized shifts'' form a class of examples inspired by Bernoulli
shifts, Bernoulli shifts are special generalized shifts and there are
several papers dedicated to compute the entropy of generalized
shifts~\cite{qm, md, nili}. ``Weighted generalized shifts'' are a further generalization of generalized shifts, 
in functional analysis and complex analysis courses 
they are known as weighted composition operators, under a different point of view
 (see subsection~\ref{najaf20}).
\\
Amongst different types of entropy, we study the covariant and the contravariant set--theoretical entropies of self--maps of a nonempty set
(see subsection~\ref{najaf10}) with a focus on the weighted generalized shifts.
\\
Let's bring our main results in this stage accompanying with their full assumptions.
For a finite field $F$, a nonempty set $\Gamma$, an arbitrary self--map $\varphi:\Gamma\to\Gamma$ and a weight vector
$\mathfrak{w}=(\mathfrak{w}_\alpha)_{\alpha\in\Gamma}\in F^\Gamma$ also
$\rm{supp}(\mathfrak{w}):=\{\alpha\in\Gamma: \mathfrak{w}_\alpha\neq0\}$, our results
on the weighted generalized shift $\mathop{\sigma_{\varphi,\mathfrak{w}}:F^\Gamma\to F^\Gamma\SP\SP}\limits_{
(x_\alpha)_{\alpha\in\Gamma}\mapsto(\mathfrak{w}_\alpha x_{\varphi(\alpha)})_{\alpha\in\Gamma}}$
 can be divided into  types ``structural results'' and ``numerical results''. As the structural results
we show:
\begin{itemize}
\item[-] $\sigma_{\varphi,\mathfrak{w}}:F^\Gamma\to F^\Gamma$ is quasi--periodic if and only if it is pointwise quasi--periodic
	(see Corollary~\ref{maryam60} also for definitions of quasi periodicity and pointwise quasi--periodicity see \ref{farangi}),
	moreover, this proposition is not valid in general self--maps by Example~\ref{kheybar40},
\item[-]  the weighted generalized shift $\sigma_{\varphi,\mathfrak{w}}:F^\Gamma\to F^\Gamma$ is of finite fibre,
	if and only if $\Gamma\setminus\varphi(\rm{supp}(\mathfrak{w}))$ is finite,
\item[-] $\sigma_{\varphi,\mathfrak{w}}(
	\mathop{\bigoplus}\limits_{\Gamma}F)\subseteq \mathop{\bigoplus}\limits_{\Gamma}F$  if and only if 
	$\varphi\restriction_{{\rm supp}(\mathfrak{w})}:{\rm supp}(\mathfrak{w})\to\Gamma$ is of finite fibre.
	Moreover, in this case, 	$\sigma_{\varphi,\mathfrak{w}}\restriction_{\mathop{\bigoplus}\limits_{\Gamma}F}$ is of finite fibre if and only if 
	$\sigma_{\varphi,\mathfrak{w}}:F^\Gamma\to F^\Gamma$ is of finite fibre,
\end{itemize}
In the numerical results, we compute covariant and 
contravariant set--theoretical entropies of $\sigma_{\varphi,\mathfrak{w}}$,
i.e.:
\begin{itemize}
\item[-] covariant set--theoretical entropy of the weighted generalized shift $\sigma_{\varphi,\mathfrak{w}}:F^\Gamma\to F^\Gamma$ has
no finite values different from zero, moreover it takes value zero if and only if $\sigma_{\varphi,\mathfrak{w}}$ is quasi--periodic,
\item[-] contravariant set--theoretical entropy of the finite fibre weighted generalized shift $\sigma_{\varphi,\mathfrak{w}}:F^\Gamma\to F^\Gamma$ has
no finite values different from zero, as a matter of fact
$\ecset(\sigma_{\varphi,\mathfrak{w}})=+\infty$ if and only if  one of the following conditions occurs:
\begin{itemize}
\item[a.] $\varphi$ has a non--quasi--periodic point in $\bigcap\{\varphi^{-n}({\rm supp}({\mathfrak w})):n\geq0\}$,
\item[b.] all points of $\bigcap\{\varphi^{-n}({\rm supp}({\mathfrak w})):n\geq0\}$ are quasi--periodic points of $\varphi$ and
	$\sup\{{\rm per}(\alpha):\alpha\in {\rm Per}(\varphi)\cap (\bigcap\{\varphi^{-n}({\rm supp}({\mathfrak w})):n\geq0\})\}=+\infty$,
\end{itemize}
\item[-]
In addition whenever $\sigma_{\varphi,\mathfrak{w}}(
	\mathop{\bigoplus}\limits_{\Gamma}F)\subseteq \mathop{\bigoplus}\limits_{\Gamma}F$, then:
\begin{itemize}
\item[a.] covariant set--theoretical entropy of 
	$\sigma_{\varphi,\mathfrak{w}}\restriction_{\mathop{\bigoplus}\limits_{\Gamma}F}:\mathop{\bigoplus}\limits_{\Gamma}F\to
	\mathop{\bigoplus}\limits_{\Gamma}F$ has no finite values different from zero, moreover it takes $+\infty$ if and only if 
	there exists a one to one $\varphi-$anti--orbit sequence in ${\rm supp}(\mathfrak{w})$,
\item[b.] contravariant set--theoretical entropy of  finite fibre
	$\sigma_{\varphi,\mathfrak{w}}\restriction_{\mathop{\bigoplus}\limits_{\Gamma}F}:\mathop{\bigoplus}\limits_{\Gamma}F\to
	\mathop{\bigoplus}\limits_{\Gamma}F$ has no finite values different from zero, moreover it takes $+\infty$ if and only if 
	there exists a one to one $\varphi-$orbit sequence in ${\rm supp}(\mathfrak{w})$.
\end{itemize}
\end{itemize}
\section{Some required backgrounds}
 \noindent Covariant set--theoretical entropy, contravariant set--theoretical entropy and weighted generalized shifts 
 are our main objects in this paper. Let's have a glance at them via the following two subsections which are followed
 by a subsection devoted to primary notations.
\subsection{Background on covariant set--theoretical and contravariant set--theoretical entropies}\label{najaf10}
For a nonempty set $A$ and a self--map $f:A\to A$, we say that the sequence
$S=\{a_n\}_{n\geq1}$ is:
\begin{itemize}
\item an \textit{$f-$orbit}, if for each $n\geq 1$, we have $f(a_n)=a_{n+1}$,
\item an \textit{$f-$anti--orbit}, if for each $n\geq 1$, we have
$f(a_{n+1})=a_n$,
\item  \textit{one to one}  if $a_m\neq a_n$ for all $m>n\geq1$.
\end{itemize}
The \textit{string number} or \textit{infinite orbit number} of self--map $f:A\to A$ is \cite{qm, cmuc}:
{\small \[\mathsf{o}(f):=\sup(\{0\}\cup\{n\geq1:{\rm there \: exist \:}n\:{\rm one \: to \: one\: pairwise\: disjoint\:}f\!-\!{\rm orbits\: in\:}A\}),\]}
and the \textit{antistring number} or \textit{infinite anti--orbit number} of self--map $f:A\to A$ is~\cite{md, cmuc}:
{\small
\[\mathsf{a}(f):=\sup(\{0\}\cup\{n\geq1:\!{\rm there \: exist \:}n \: {\rm one\: to \: one\: pairwise\: disjoint\:}f\!-\!{\rm anti\! -\!orbits\: in\:}A\}).\]}
 For a finite subset $D$ of $A$ let (whereby $\vert D\vert$ we mean the cardinality of $D$):
\[\mathfrak{h}(f,D):={\displaystyle\lim_{n\to\infty}\frac{\vert D\cup f(D)\cup f^2(D)\cup\cdots\cup f^{n-1}(D)\vert}{n}}\]
where the limit exists by \cite[Lemma 2.6]{md}, and we call $\eset(f):=\sup\{\mathfrak{h}(f,D):D$ is a
finite subset of $A\}$, \textit{covariant set--theoretical entropy} of $f:A\to A$ (or briefly 
\textit{set--theoretical entropy} of $f:A\to A$),
moreover, $\eset(f)=\mathsf{o}(f)$~\cite[Proposition 2.16]{md}. Set--theoretical entropy of a
self--map has been introduced for the first time in~\cite{md}.
\\
On the other hand for $f:A\to A$, we have $\bigcap\{f^n(A):n\geq1\}=\bigcup\{D\subseteq A:f(D)=D\}$ so 
 ${\rm
sc}(f)=\bigcap\{f^n(A):n\geq1\}$ as \textit{surjective core} of $f:A\to
A$  is the biggest subset $D$ of $A$, such that
$f\restriction_D:D\to D$ is surjective. 
Also the sequence $\{a_n\}_{n\geq1}$ is an $f-$anti--orbit if and only if it is an 
$f\restriction_{{\rm
sc}(f)}-$anti--orbit, hence $\mathsf{a}(f\restriction_{{\rm
sc}(f)})=\mathsf{a}(f)$.
\\
Let's mention that $f:A\to B$ is of \textit{finite fibre} if $f^{-1}(y)$ is finite for all $y\in B$. 
\\
For a surjective finite fibre
$f:A\to A$ and a finite subset $D$ of  $A$, let
\[\mathfrak{h}^*(f,D):={\displaystyle\limsup_{n\to\infty}\frac{\vert D\cup f^{-1}(D)\cup f^{-2}(D)\cup\cdots\cup f^{-n+1}(D)\vert}{n}}\:,\]
then we call $\ecset(f):=\sup\{\mathfrak{h}^*(f,D):D$ is a finite subset
of $A\}$, \textit{contravariant set--theoretical entropy} of $f$, moreover 
$\ecset(f)=\mathsf{a}(f)$. So for a finite fibre self--map $f:A\to A$, 
we may consider  contravariant set--theoretical entropy of $f$
 as
$\ecset(f):=\mathsf{a}(f\restriction_{{\rm
sc}(f)})=\mathsf{a}(f)$, see~\cite[Definition 3.2.18, Definition 3.2.19, Proposition 3.2.34, Theorem 3.2.39]{anna}.
\subsection{Background on generalized and weighted generalized shifts}\label{najaf20}
Let's recall that for the nonempty sets $M,\Gamma$ and the self--map $\varphi :\Gamma\to\Gamma$,
the \textit{generalized shift}
$\sigma_\varphi:M^\Gamma\to M^\Gamma$ with
$\sigma_\varphi((x_\alpha)_{\alpha\in\Gamma})=(x_{\varphi(\alpha)})_{\alpha\in\Gamma}$
has been introduced for the first time in~\cite{note} as a
generalization of the left Bernoulli shift and the two--sided shift.
Moreover, for a bounded vector $(r_n)_{n\geq1}(\in{\mathbb C}^{\mathbb N})$, the weighted shift
$\sigma:\ell^2\to\ell^2$ with $\sigma((x_n)_{n\geq1})=(r_nx_{n+1})_{n\geq1}$ is of great interest in functional analysis.
\\
Weighted generalized shifts can be considered as a common generalization of weighted shifts and generalized shifts in the following way:
Suppose $M$ is a module over ring $R$, $\Gamma$ is a nonempty
set, $\varphi:\Gamma\to\Gamma$ is an arbitrary self--map, and
$\mathfrak{w}=(\mathfrak{w}_\alpha)_{\alpha\in\Gamma}\in
R^\Gamma$, then we call
$\sigma_{\varphi,\mathfrak{w}}:M^\Gamma\to M^\Gamma$ with
$\sigma_{\varphi,\mathfrak{w}}((x_\alpha)_{\alpha\in\Gamma})=(\mathfrak{w}_\alpha
x_{\varphi(\alpha)})_{\alpha\in\Gamma}$ a \textit{weighted generalized
shift}~\cite{lp}, moreover $\sigma_{\varphi,\mathfrak{w}}={\mathfrak w}\sigma_\varphi$. Also if $1$ is the unit element of $R$ 
and $\mathfrak{w}=(1)_{\alpha\in\Gamma}$, then
$\sigma_{\varphi,\mathfrak{w}}=\sigma_\varphi$.
\\
For the connections between topological (algebraic) entropy and set-- theoretical entropies in generalized shifts see \cite{cmuc}.
In other point of view for computing  the topological, algebraic and set--theoretical entropies of a generalized shift see~\cite{qm, md, gior, nili}. 
\subsection{Primary notations}\label{farangi}
Let's recall that for $f:A\to A$,
\begin{itemize}
\item ${\rm Fix}(f)=\{x\in A:f(x)=x\}$ is the collection of \textit{fixed points} of $f:A\to A$,
\item ${\rm Per}(f)=\{x\in A:\exists n\geq1 \:\:f^n(x)=x\}$ is the collection of  \textit{periodic points} of $f:A\to A$,
\item if $x\in {\rm Per}(f)$, then ${\rm per}(x):=\min\{n\geq1:f^n(x)=x\}$ is the \textit{period} of $x$ (w.r.t. $f$),
\item ${\rm QPer}(f)=\{x\in A:\exists m>n\geq1 \:\:f^n(x)=f^m(x)\}$ is the collection of \textit{quasi--periodic points} of $f:A\to A$,
\item $A\setminus {\rm QPer}(f)$ is the collection of \textit{non--quasi--periodic points} of $f:A\to A$,
\item $f$ is \textit{periodic} if there exists $n\geq1$ with $f^n=id_A$,
\item $f$ is \textit{quasi--periodic} if there exists $n>m\geq1$ with $f^n=f^m$,
\item $f$ is \textit{pointwise periodic} (resp. \textit{pointwise quasi--periodic}) if ${\rm Per}(f)=A$ (resp. ${\rm QPer}(f)=A$).
\end{itemize}
Also note that ${\rm Fix}(f)\subseteq{\rm Per}(f)\subseteq {\rm QPer}(f)$. 
\\
Moreover, $z\in A\setminus {\rm QPer}(f)$ if and only if for all distinct $p,q\geq1$ we have
$f^p(z)\neq f^q(z)$, i.e. $\{f^n(z)\}_{n\geq1}$ is one to one. Therefore $A\setminus {\rm QPer}(f)=\{x\in A:\{f^n(x)\}_{n\geq1}$ is
one to one$\}$. 
\\
Now we show $A\setminus {\rm QPer}(f)=\{x\in A:\{f^n(x)\}_{n\geq1}$ is
infinite$\}$. 
Let $z\in A\setminus {\rm QPer}(f)$, then $\{f^n(z)\}_{n\geq1}$ is one to one, therefore $\{f^n(z)\}_{n\geq1}$ is infinite.
So $A\setminus {\rm QPer}(f)\subseteq \{x\in A:\{f^n(x)\}_{n\geq1}$ is
infinite$\}$. On the other hand if $t\in {\rm QPer}(f)$, choose $p>q\geq1$ with $f^p(t)=f^q(t)$, then 
$\{f^n(t):n\geq1\}=\{f^n(t):1\leq n\leq p\}$ in particular $\{f^n(t)\}_{n\geq1}$ is finite and
${\rm QPer}(f)\subseteq \{x\in A:\{f^n(x)\}_{n\geq1}$ is
finite$\}$. Therefore $\{x\in A:\{f^n(x)\}_{n\geq1}$ is
infinite$\}\subseteq A\setminus {\rm QPer}(f)$.
\begin{exam}\label{kheybar40}
Consider $\eta:\mathbb{N}\to\mathbb{N}$ with 
$\eta(1)=1$, $\eta(j)=j+1$ for $n(n+1)/2<j<(n+1)(n+2)/2$ and $\eta((n+1)(n+2)/2)=n(n+1)/2+1$. So:
\[\begin{array}{l}
\eta(1)=1\:, \\
\eta(2)=3,\:\eta(3)=2,\\
\eta(4)=5,\:\eta(5)=6,\:\eta(6)=4, \\
\vdots 
\end{array}\]
is pointwise periodic and so also pointwise quasi--periodic, however, it is neither periodic nor quasi--periodic.
\end{exam}
\begin{convention}\label{maryam10}
In the following text consider finite field a $F$, a nonempty set
$\Gamma$, a self--map $\varphi:\Gamma\to\Gamma$, a \textit{weight vector}
$\mathfrak{w}=(\mathfrak{w}_\alpha)_{\alpha\in\Gamma}\in F^\Gamma$
and the weighted generalized shift
$\sigma_{\varphi,\mathfrak{w}}:F^\Gamma\to F^\Gamma$ (so
$\sigma_{\varphi,\mathfrak{w}}(x_\alpha)_{\alpha\in\Gamma}=(\mathfrak{w}_\alpha
x_{\varphi(\alpha)})_{\alpha\in\Gamma}$ for each
$(x_\alpha)_{\alpha\in\Gamma}\in F^\Gamma$). For
$L\subseteq\Gamma$ and $x=(x_\alpha)_{\alpha\in\Gamma}\in
F^\Gamma$, let 
\[x^L:=(x_\alpha)_{\alpha\in L}\:.\]
Also, let (where
$\varphi^0=id_\Gamma:\mathop{\Gamma\to\Gamma}\limits_{\alpha\mapsto\alpha}$ is the identity map on $\Gamma$):
\[\mathcal{T}:=\{(n,\alpha):n\geq0,\alpha\in\Gamma, | \{\varphi^i(\alpha):0\leq i\leq n\} |=n+1 ,
\mathop{\prod}\limits_{0\leq i\leq n}\mathfrak{w}_{\varphi^i(\alpha)}\neq0\}.\]
In $\mathcal T$ we search for $(n,\alpha)$s such that 
\[\{(\mathfrak{w}_\alpha x_\alpha,\mathfrak{w}_{\varphi(\alpha)} x_{\varphi(\alpha)},
\cdots,\mathfrak{w}_{\varphi^n(\alpha)} x_{\varphi^n(\alpha)}):x_\alpha,x_{\varphi(\alpha)},\ldots,x_{\varphi^n(\alpha)}\in F\}\]
forms  a linear vector space over
$F$ of dimension $n+1$. Moreover if  $(n,\alpha)\in\mathcal{T}$, then
$(i,\alpha)\in\mathcal{T}$ for each $i\in\{0,1,\ldots,n\}$.
\end{convention}
\noindent In this paper Convention~\ref{maryam10}  is valid for whole of the text, Convention~\ref{maryam20} 
is valid for subsection~\ref{maryam30} and Convention~\ref{kheybar75} is valid for some parts of Section~\ref{maryam40} and
whole of Section~\ref{maryam50}. 
\section{Set--theoretical entropy of $\sigma_{\varphi,\mathfrak{w}}:F^\Gamma\to F^\Gamma$}
\noindent In this section we prove $\eset(\sigma_{\varphi,\mathfrak{w}})\in\{0,+\infty\}$, moreover
$\eset(\sigma_{\varphi,\mathfrak{w}})=0$ if and only if $\sigma_{\varphi,\mathfrak{w}}$ is quasi--periodic.
Hence, $\sigma_{\varphi,\mathfrak{w}}$ is pointwise quasi--periodic if and only if it is quasi--periodic.
Note that, for $x=(x_\alpha)_{\alpha\in\Gamma},y=(y_\alpha)_{\alpha\in\Gamma}\in
F^\Gamma$ we have $xy=(x_\alpha y_\alpha)_{\alpha\in\Gamma}$.
\begin{lemma}\label{kheybar50}
Consider $\mathfrak{u}=({\mathfrak u}_\alpha)_{\alpha\in\Gamma}\in F^\Gamma$ and $\psi:\Gamma\to\Gamma$,
then $\sigma_{\psi,\mathfrak{u}}\circ \sigma_{\varphi,\mathfrak{w}}=\sigma_{\varphi\circ\psi,\mathfrak{u}\sigma_\psi(\mathfrak{w})}$.
\\
Also for all $n\geq1$, we have
$\sigma_{\varphi,\mathfrak{w}}^n=\sigma_{\varphi^n,\mathfrak{w}\sigma_\varphi(\mathfrak{w})\cdots\sigma_{\varphi^{n-1}}(\mathfrak{w})}$,
i.e.
\begin{equation}\label{tazeh11}
\forall(x_\alpha)_{\alpha\in\Gamma}\in F^\Gamma,\:\:\sigma_{\varphi,\mathfrak{w}}^n((x_\alpha)_{\alpha\in\Gamma})=
(\mathfrak{w}_\alpha\mathfrak{w}_{\varphi(\alpha)}\cdots\mathfrak{w}_{\varphi^{n-1}(\alpha)}x_{\varphi^n(\alpha)})_{\alpha
\in\Gamma}\:.
\end{equation}
On the other hand, $\sigma_{\varphi,\mathfrak{w}}(z)={\mathfrak w}\sigma_\varphi(z)$ (for all $z\in F^\Gamma$).
\end{lemma}
\begin{proof}
For all $z=(z_\alpha)_{\alpha\in\Gamma}\in F^\Gamma$ we have:
\begin{eqnarray}
\sigma_{\psi,\mathfrak{u}}\circ \sigma_{\varphi,\mathfrak{w}}(z) & = &
    \sigma_{\psi,\mathfrak{u}}(\sigma_{\varphi,\mathfrak{w}}((z_\alpha)_{\alpha\in\Gamma}))=\sigma_{\psi,\mathfrak{u}}(
    (\mathfrak{w}_\alpha z_{\varphi(\alpha)})_{\alpha\in\Gamma}) \\
& = & (\mathfrak{u}_\alpha \mathfrak{w}_{\psi(\alpha)} z_{\varphi(\psi(\alpha))})_{\alpha\in\Gamma}\label{33}
\\
& = & 
    \sigma_{\varphi\circ\psi,\mathfrak{u}\sigma_\psi(\mathfrak{w})}((z_\alpha)_{\alpha\in\Gamma})
    =\sigma_{\varphi\circ\psi,\mathfrak{u}\sigma_\psi(\mathfrak{w})}(z)\:.
\end{eqnarray}
Regarding the equality in \ref{33}, note that if $c_\alpha=\mathfrak{w}_\alpha z_{\varphi(\alpha)}$ for each $\alpha\in\Gamma$, then
$c_{\psi(\alpha)}=\mathfrak{w}_{\psi(\alpha)} z_{\varphi(\psi(\alpha))}$ and
$\mathfrak{u}_\alpha c_{\psi(\alpha)}=\mathfrak{u}_\alpha\mathfrak{w}_{\psi(\alpha)} z_{\varphi(\psi(\alpha))}$
for each $\alpha\in\Gamma$.
\\
Therefore:
\[\sigma_{\psi,\mathfrak{u}}\circ \sigma_{\varphi,\mathfrak{w}}(z)=\mathfrak{u}\sigma_\psi(\mathfrak{w})\sigma_{\varphi\circ
\psi}(z)=\sigma_{\varphi\circ\psi,\mathfrak{u}\sigma_\psi(\mathfrak{w})}(z)\:.\]
Use induction on $n\geq1$ to obtain \ref{tazeh11}.
\end{proof}
In the following lemma, we characterize the quasi--periodic weighted generalized shifts in terms of ${\mathcal T}$.
\begin{lemma}\label{salam10}
The following statements are equivalent:
\begin{itemize}
\item[1.] $\sigma_{\varphi,\mathfrak{w}}:F^\Gamma\to F^\Gamma$ is quasi--periodic,
\item[2.] there exist  $1\leq n<m$ such that, for any $\alpha\in\Gamma$, we have $\varphi^n(\alpha)=\varphi^m(\alpha)$ or 
$\mathfrak{w}_\alpha
    \mathfrak{w}_{\varphi(\alpha)}\cdots\mathfrak{w}_{\varphi^{n-1}(\alpha)}=0$,
\item[3.] $\sup(\{n:\exists\alpha\in\Gamma\:\:(n,\alpha)\in{\mathcal T}\}\cup\{0\})<+\infty$.
\end{itemize}
\end{lemma}
\begin{proof}
(1 $\Rightarrow$ 2): For $m>n\geq1$, suppose
$\sigma^n_{\varphi,\mathfrak{w}}=\sigma^m_{\varphi,\mathfrak{w}}$ and consider $\beta\in\Gamma$
such that $\varphi^n(\beta)\neq\varphi^m(\beta)$. Choose $(x_\alpha)_{\alpha\in\Gamma}\in
F^\Gamma$ such that $x_{\varphi^n(\beta)}=1$ and
$x_{\varphi^m(\beta)}=0$, then
\begin{eqnarray*}
(\mathfrak{w}_\alpha\mathfrak{w}_{\varphi(\alpha)}\cdots\mathfrak{w}_{\varphi^{n-1}(\alpha)}x_{\varphi^n(\alpha)})_{\alpha\in\Gamma}
	& = & \sigma^n_{\varphi,\mathfrak{w}}((x_\alpha)_{\alpha\in\Gamma}) = \sigma^m_{\varphi,\mathfrak{w}}((x_\alpha)_{\alpha\in\Gamma}) \\
& = &
    ( \mathfrak{w}_\alpha\mathfrak{w}_{\varphi(\alpha)}\cdots\mathfrak{w}_{\varphi^{m-1}(\alpha)}x_{\varphi^m(\alpha)})_{\alpha\in\Gamma}
\end{eqnarray*}
in particular
\[\mathfrak{w}_\beta\mathfrak{w}_{\varphi(\beta)}\cdots\mathfrak{w}_{\varphi^{n-1}(\beta)}x_{\varphi^n(\beta)}=
\mathfrak{w}_\beta\mathfrak{w}_{\varphi(\beta)}\cdots\mathfrak{w}_{\varphi^{m-1}(\beta)}x_{\varphi^m(\beta)}\]
using $x_{\varphi^n(\beta)}=1$ and
$x_{\varphi^m(\beta)}=0$ we conclude that $\mathfrak{w}_\beta\mathfrak{w}_{\varphi(\beta)}\cdots\mathfrak{w}_{\varphi^{n-1}(\beta)}=0$.
\\
(2 $\Rightarrow$ 3): If (2) holds, then
$\sup(\{k:\exists\alpha\in\Gamma\:\:(k,\alpha)\in{\mathcal
T}\}\cup\{0\})\leq m-1<+\infty$.
\\
(3 $\Rightarrow$ 1): Suppose
$\sup(\{k:\exists\alpha\in\Gamma\:\:(k,\alpha)\in{\mathcal
T}\}\cup\{0\})$ is finite, let
\[p=\sup(\{k:\exists\alpha\in\Gamma\:\:(k,\alpha)\in{\mathcal
T}\}\cup\{0\})+1\:\:{\rm and}\:\: q=(p+1)!\:.\]
We prove that $\sigma_{\varphi,\mathfrak{w}}$ is quasi--periodic
via the following 2 claims.
\\
{\it Claim A.}
For $\beta\in\Gamma$ and for some $q\geq1$
if $\varphi^q(\beta)=\beta$, 
then 
\[\mathop{\prod}\limits_{0\leq i\leq \vert F\vert q-1}\mathfrak{w}_{\varphi^i(\beta)}=
\mathop{\prod}\limits_{0\leq i\leq q-1}\mathfrak{w}_{\varphi^i(\beta)} \:.\]
{\it Proof of Claim A.} 
Note  that $F\setminus\{0\}$ is a multiplicative group with
$|F|-1$ elements and identity $1$, thus $x^{|F|-1}=1$ for each $x\in F\setminus\{0\}$, hence $x^{|F|}=x$ for all $x\in F$, so
\[\begin{array}{rcl}
\mathop{\prod}\limits_{0\leq i\leq \vert F\vert q-1}\mathfrak{w}_{\varphi^i(\beta)} & = &
    \mathop{\prod}\limits_{0\leq j\leq \vert F\vert -1}\bigg(
    \mathop{\prod}\limits_{qj\leq i\leq q(j+1)-1}\mathfrak{w}_{\varphi^i(\beta)}\bigg)  \\ && \\
& = &   \mathop{\prod}\limits_{0\leq j\leq \vert F\vert -1}\bigg(
    \mathop{\prod}\limits_{0\leq i\leq q-1}\mathfrak{w}_{\varphi^{qj+i}(\beta)}\bigg)  \\ && \\
& = &
    \mathop{\prod}\limits_{0\leq j\leq \vert F\vert -1}\bigg(
    \mathop{\prod}\limits_{0\leq i\leq q-1}\mathfrak{w}_{\varphi^i(\beta)}\bigg) \\
    && \\
& = & \bigg(\mathop{\prod}\limits_{0\leq i\leq q-1}\mathfrak{w}_{\varphi^i(\beta)}\bigg )^{\vert F\vert} 
    = \mathop{\prod}\limits_{0\leq i\leq q-1}\mathfrak{w}_{\varphi^i(\beta)} 
\end{array}
\]
{\it Claim B.} Let $\alpha\in\Gamma$ and $x=(x_\theta)_{\theta\in\Gamma}\in F^\Gamma$, then:
\[\bigg(\mathop{\prod}\limits_{0\leq i\leq q+p}\mathfrak{w}_{\varphi^i(\alpha)}\bigg)x_{\varphi^{q+p+1}(\alpha)}=
\bigg(\mathop{\prod}\limits_{0\leq i\leq \vert F\vert q+p}\mathfrak{w}_{\varphi^i(\alpha)}\bigg)x_{\varphi^{\vert F \vert q+p+1}(\alpha)}\:.\]
{\it Proof of Claim B.}
Using the definition of $p$, for every $\alpha\in\Gamma$, $(p+1,\alpha)\notin\mathcal{T} $. By the definition of $\mathcal T$
and  $(p+1,\alpha)\notin\mathcal{T} $, we have:
\[\mathfrak{w}_\alpha\mathfrak{w}_{\varphi(\alpha)}\cdots\mathfrak{w}_{\varphi^{p+1}(\alpha)}=0 \vee
    \vert \{\alpha,\varphi(\alpha),\ldots,\varphi^{p+1}(\alpha)\}\vert <p+2\:.\]
If $\mathfrak{w}_\alpha\mathfrak{w}_{\varphi(\alpha)}\cdots\mathfrak{w}_{\varphi^{p+1}(\alpha)}=0$, then
\[0=\bigg(\mathop{\prod}\limits_{0\leq i\leq \vert F\vert q+p}\mathfrak{w}_{\varphi^i(\alpha)}\bigg)x_{\varphi^{\vert F \vert q+p+1}(\alpha)}=\bigg(\mathop{\prod}\limits_{0\leq i\leq q+p}\mathfrak{w}_{\varphi^i(\alpha)}\bigg)x_{\varphi^{q+p+1}(\alpha)}\:.\]
On the other hand, if $\vert \{\alpha,\varphi(\alpha),\ldots,\varphi^{p+1}(\alpha)\}\vert <p+2$, then there exists $0\leq i<j\leq p+1$ such that 
$\varphi^i(\alpha)=\varphi^j(\alpha)$, thus $\varphi^i(\alpha)\in{\rm Per}(\varphi)$ with ${\rm per}(\varphi^i(\alpha))\leq j-i\leq p+1$.
Since $\varphi({\rm Per}(\varphi))={\rm Per}(\varphi)$, so $\varphi^{p+1}(\alpha)=\varphi^{(p+1)-i}(\varphi^i(\alpha))\in {\rm Per}(\varphi)$.
Moreover, ${\rm per}(\varphi^{p+1}(\alpha))={\rm per}(\varphi^{(p+1)-i}(\varphi^i(\alpha)))={\rm per}(\varphi^i(\alpha))\leq p+1$,
therefore  ${\rm per}(\varphi^{p+1}(\alpha))$ divides $q=(p+1)!$ which shows $\varphi^{q+p+1}(\alpha)=\varphi^{p+1}(\alpha)$.
By Claim A we have
$\mathop{\prod}\limits_{0\leq i\leq \vert F\vert q-1}\mathfrak{w}_{\varphi^i(\varphi^{p+1}(\alpha))} 
= \mathop{\prod}\limits_{0\leq i\leq q-1}\mathfrak{w}_{\varphi^i(\varphi^{p+1}(\alpha))} $ thus:
\[\mathop{\prod}\limits_{p+1\leq i\leq \vert F\vert q+p}\mathfrak{w}_{\varphi^i(\alpha)}=
\mathop{\prod}\limits_{0\leq i\leq \vert F\vert q-1}\mathfrak{w}_{\varphi^{i+p+1}(\alpha)} 
= \mathop{\prod}\limits_{0\leq i\leq q-1}\mathfrak{w}_{\varphi^{i+p+1}(\alpha)} =\mathop{\prod}\limits_{p+1\leq i\leq q+p}\mathfrak{w}_{\varphi^i(\alpha)}\]
therefore $\mathop{\prod}\limits_{0\leq i\leq \vert F\vert q+p}\mathfrak{w}_{\varphi^i(\alpha)}=
    \mathop{\prod}\limits_{0\leq i\leq q+p}\mathfrak{w}_{\varphi^i(\alpha)}$. So
    (use $\varphi^{q+p+1}(\alpha)=\varphi^{\vert F \vert q+p+1}(\alpha)$):
    \[\bigg(\mathop{\prod}\limits_{0\leq i\leq q+p}\mathfrak{w}_{\varphi^i(\alpha)}\bigg)x_{\varphi^{q+p+1}(\alpha)}=
\bigg(\mathop{\prod}\limits_{0\leq i\leq \vert F\vert q+p}\mathfrak{w}_{\varphi^i(\alpha)}\bigg)x_{\varphi^{\vert F \vert q+p+1}(\alpha)}\:,\]
which completes the proof of Claim B.
\\
Now we are ready to complete the proof of Lemma~\ref{salam10}. Indeed, by Claim B we have
$\sigma_{\varphi,\mathfrak{w}}^{q+p+1}=\sigma_{\varphi,\mathfrak{w}}^{q \vert F\vert +p+1}$
which means that $\sigma_{\varphi,\mathfrak{w}}:F^\Gamma\to F^\Gamma$ is quasi--periodic.
\end{proof}
\noindent In order to have intuition on weighted generalized shifts with infinite set-theoretical entropy in the
following example we 
bring a one to one map $\eta_1:\mathbb{N}\to\mathbb{N}$ without any periodic point and a pointwise quasi--periodic
$\eta_2:\mathbb{N}\to\mathbb{N}$ such that for $\mathfrak{u}:=(1)_{n\in\mathbb{N}}$ both maps
$\sigma_{\eta_1,\mathfrak{u}},\sigma_{\eta_2,\mathfrak{u}}:F^\mathbb{N}\to F^\mathbb{N}$ has infinite
set-theoretical entropy.
\begin{exam}
Consider $\eta_1,\eta_2:\mathbb{N}\to\mathbb{N}$ with $\eta_1(n)=n+1$ (for $n\in\mathbb{N}$) and $\eta_2=\eta$ as in
Example~\ref{kheybar40}. Also let:
\[\begin{array}{lcl}
z_1& := & (1,\:\: 1,0,\;\: 1,0,0,\:\: 1,0,0,0,\:\: 1,0,0,0,0,\:\:1,0,0,0,0,0,\:\:\cdots)\:,\\
z_2& := & (1,\:\: 1,1,\;\: 1,1,0,\:\: 1,1,0,0,\:\: 1,1,0,0,0,\:\:1,1,0,0,0,0,\:\:\cdots)\:,\\
z_3& := & (1,\:\: 1,1,\;\: 1,1,1,\:\: 1,1,1,0,\:\: 1,1,1,0,0,\:\:1,1,1,0,0,0,\:\:\cdots)\:,\\
&\vdots & 
\end{array}\]
then for $\mathfrak{u}:=(1)_{n\in\mathbb{N}}$, the sequences 
$\{\sigma_{\eta_i,\mathfrak{u}}^n(z_1)\}_{n\geq1},\{\sigma_{\eta_i,\mathfrak{u}}^n(z_2)\}_{n\geq1},\{\sigma_{\eta_i,\mathfrak{u}}^n(z_3)\}_{n\geq1},\ldots$ are pairwise disjoint one to one sequences. 
\end{exam}
\begin{theorem}\label{ssalam10}
For the weighted generalized shift $\sigma_{\varphi,\mathfrak{w}}:F^\Gamma\to F^\Gamma$ we have:
\[\eset(\sigma_{\varphi,\mathfrak{w}})=\left\{\begin{array}{lc} 0, & {\rm if \:} \sigma_{\varphi,\mathfrak{w}} {\rm \: is \: quasi-periodic}, \\
    +\infty, & {\rm otherwise\:}. \end{array}\right.\]
\end{theorem}
\begin{proof}
We already know that the quasi--periodicity $\sigma_{\varphi,\mathfrak{w}}$  implies the pointwise quasi--periodicity of
$\sigma_{\varphi,\mathfrak{w}}$, which means
that $\mathsf{o}(\sigma_{\varphi,\mathfrak{w}})=\eset(\sigma_{\varphi,\mathfrak{w}})=0$. So it remains to prove that 
non--quasi--periodicity of $\sigma_{\varphi,\mathfrak{w}}$ implies
$\eset(\sigma_{\varphi,\mathfrak{w}})=+\infty$.
\\
Suppose $\sigma_{\varphi,\mathfrak{w}}:F^\Gamma\to F^\Gamma$ is not quasi--periodic, then by Lemma~\ref{salam10} we have
\linebreak
$\sup(\{k:\exists\alpha\in\Gamma\:\:(k,\alpha)\in{\mathcal T}\}\cup\{0\})=+\infty$. We aim to prove that
$\eset(\sigma_{\varphi,\mathfrak{w}})=+\infty$.
Using induction choose
a sequence $\{\alpha_i\}_{i\geq1}$ ($\subseteq\Gamma$) in the following way:
\begin{itemize}
\item there exists $\alpha_1$ such that  $(n_1,\alpha_1)=(1,\alpha_1)\in\mathcal{T}$,
\item for $k\geq1$, suppose $(1,\alpha_1),\ldots,(k,\alpha_k)\in\mathcal{T}$ have been chosen
such that $\{\alpha_1,\varphi(\alpha_1)\},\ldots, \{\alpha_k,\varphi(\alpha_k),\ldots,\varphi^k(\alpha_k)\}$ are pairwise disjoint.
There exists
    $(n_{k+1},\beta)\in\mathcal{T}$ with $n_{k+1}\geq (k+1)(k+2)+\underbrace{\frac{(k+1)(k+2)}{2}-1}_{2+3+\cdots+(k+1)}$.
Suppose
\[H:=\{\beta,\varphi(\beta),\ldots,\varphi^{n_{k+1}}(\beta)\}\setminus(\cup\{\{\alpha_i,\varphi(\alpha_i),\ldots,
    \varphi^i(\alpha_i)\}:1\leq i\leq k\})\]
is equal to:
$\{\varphi^{s_1}(\beta),\varphi^{s_1+1}(\beta),\ldots,\varphi^{s_1+t_1}(\beta)\}\cup \{\varphi^{s_2}(\beta),\ldots,\varphi^{s_2+t_2}(\beta)\}\cup\cdots
\cup \{\varphi^{s_p}(\beta),\ldots,\varphi^{s_p+t_p}(\beta)\}$
with
\begin{center}
$0\leq s_1\leq s_1+t_1<s_2-1<s_2+t_2<\cdots<s_p-1<s_p+t_p\leq n_{k+1}$,
\end{center}
then $p\leq k+1$ and
\begin{eqnarray*}
\vert H\vert & = & (t_1+1)+(t_2+1)+\cdots+(t_p+1) \\
    & \geq & n_{k+1}+1-\vert \cup\{\{\alpha_i,\varphi(\alpha_i),\ldots, \varphi^i(\alpha_i)\}:1\leq i\leq k\}\vert \\
    & \geq & n_{k+1}+1-(2+\cdots+(k+1))\geq (k+1)(k+2)
\end{eqnarray*}
since $\vert H\vert=(t_1+1)+(t_2+1)+\cdots+(t_p+1)\geq (k+1)(k+2) $ and $p\leq k+1$
there exists $j\in\{1,\ldots, p\}$ with $t_j+1\geq k+2$, thus:
{\small
\begin{equation}\label{tazeh}
\begin{array}{c}
\\
\{\varphi^{s_j}(\beta),\varphi^{s_j+1}(\beta),\ldots,\varphi^{s_j+(k+1)}(\beta)\}\cap\bigg(\cup\{\{\alpha_i,\varphi(\alpha_i),\ldots,
    \varphi^i(\alpha_i)\}:1\leq i\leq k\}\bigg) \\
\subseteq\{\varphi^{s_j}(\beta),\ldots,\varphi^{s_j+t_j}(\beta)\}\cap\bigg(\cup\{\{\alpha_i,\varphi(\alpha_i),\ldots,
    \varphi^i(\alpha_i)\}:1\leq i\leq k\}\bigg)=\varnothing \\
  \end{array}
    \end{equation}}
 let $\alpha_{k+1}:=\varphi^{s_j}(\beta)$, then $(k+1,\alpha_{k+1})\in\mathcal{T}$ and by
\ref{tazeh}
\[\{\alpha_1,\varphi(\alpha_1)\},\ldots, \{\alpha_k,\varphi(\alpha_k),\ldots,\varphi^k(\alpha_k)\},
\{\alpha_{k+1},\varphi(\alpha_{k+1}),\ldots,\varphi^{k+1}(\alpha_{k+1})\}\]
are pairwise disjoint sets.
\end{itemize}
Using the above inductive construction
$\{\{\varphi^i(\alpha_n):0\leq i\leq n\}:n\geq1\}$ is a
collection of pairwise disjoint sets and for all $n\geq1$,
$i\in\{0,\ldots,n\}$, we have $(n,\alpha_n)\in\mathcal{T}$ in particular
$\mathfrak{w}_{\varphi^i(\alpha_n)}\neq0$.
\\
For $m\geq1$, suppose $p_m$ is the $m$th prime number and let:
\[x^m_\alpha:=\left\{\begin{array}{lc} 1, & {\rm if \:} \alpha=\varphi^{p_m^n}(\alpha_{p_m^n}) {\rm \: for \: some \:} n\geq1, \\ 0, & {\rm otherwise\:,} \end{array}\right.\SP
{\rm and}\SP x^m:=(x^m_\alpha)_{\alpha\in\Gamma}.\] For
convenience let
\[\sigma_{\varphi,\mathfrak{w}}^i(x^j)=(y^{j,i}_\alpha)_{\alpha\in\Gamma}\SP(i,j\geq1)\:.\]
We claim that
$\{\{\sigma_{\varphi,\mathfrak{w}}^n(x^m)\}_{n\geq1}:m\geq1\}$ is
a collection of pairwise disjoint one to one sequences. For this
aim, consider $(m,n),(s,t)\in{\mathbb N}\times{\mathbb N}$ with
$(m,n)\neq(s,t)$, we show:
\[\sigma_{\varphi,\mathfrak{w}}^n(x^m)\neq \sigma_{\varphi,\mathfrak{w}}^t(x^s)\]
using the following cases:
\\
{\bf Case 1.} $m\neq s$. Without any loss of generality we may suppose $n\geq t$. Choose $k\geq1$ with
    $p_m^k>n$, then
    \begin{eqnarray*}
    y_{\varphi^{p_m^k-n}(\alpha_{p_m^k})}^{m,n} & = &
    \mathfrak{w}_{\varphi^{p_m^k-n}(\alpha_{p_m^k})}
    \mathfrak{w}_{\varphi^{p_m^k-n+1}(\alpha_{p_m^k})}\cdots\mathfrak{w}_{\varphi^{p_m^k-1}(\alpha_{p_m^k})}
    \underbrace{x^m_{\varphi^{p_m^k}(\alpha_{p_m^k})}}_{1} \\
    & = & \mathfrak{w}_{\varphi^{p_m^k-n}(\alpha_{p_m^k})}\cdots\mathfrak{w}_{\varphi^{p_m^k-1}(\alpha_{p_m^k})}
    \neq0\:.
    \end{eqnarray*}
Also (use the way of choosing $\alpha_i$s)
\begin{eqnarray*}
p_m^k>n\geq t & \Rightarrow & p_m^k+t-n\in\{1,\ldots,p_m^k\} \\
& \Rightarrow & \forall r\geq1\SP\varphi^{p_s^r}(\alpha_{p_s^r})\neq \varphi^{p_m^k+t-n}(\alpha_{p_m^k}) \\
& \Rightarrow & x^s_{\varphi^{p_m^k+t-n}(\alpha_{p_m^k})}=0\:.
\end{eqnarray*}
And
    \[y_{\varphi^{p_m^k-n}(\alpha_{p_m^k})}^{s,t} =
    \mathfrak{w}_{\varphi^{p_m^k-n}(\alpha_{p_m^k})}\cdots\mathfrak{w}_{\varphi^{p_m^k-n+t-1}(\alpha_{p_m^k})}
    \underbrace{x^s_{\varphi^{p_m^k+t-n}(\alpha_{p_m^k})}}_{0} =0\:.\]
    So $y_{\varphi^{p_m^k-n}(\alpha_{p_m^k})}^{m,n}\neq y_{\varphi^{p_m^k-n}(\alpha_{p_m^k})}^{s,t}$
    and $\sigma_{\varphi,\mathfrak{w}}^n(x^m)\neq \sigma_{\varphi,\mathfrak{w}}^t(x^s)$.
\\
{\bf Case 2.} $m=s$ and $n\neq t$. We may suppose $n>t$. Then:
\begin{eqnarray*}
    y_{\varphi^{p_m^n-n}(\alpha_{p_m^n})}^{m,n} & = &
    \mathfrak{w}_{\varphi^{p_m^n-n}(\alpha_{p_m^n})}
    \mathfrak{w}_{\varphi^{p_m^n-n+1}(\alpha_{p_m^n})}\cdots\mathfrak{w}_{\varphi^{p_m^n-1}(\alpha_{p_m^n})}
    \underbrace{x^m_{\varphi^{p_m^n}(\alpha_{p_m^n})}}_{1} \\
    & = & \mathfrak{w}_{\varphi^{p_m^n-n}(\alpha_{p_m^n})}\cdots\mathfrak{w}_{\varphi^{p_m^n-1}(\alpha_{p_m^n})}
    \neq0
    \end{eqnarray*}
and
    \begin{eqnarray*}
    y_{\varphi^{p_m^n-n}(\alpha_{p_m^n})}^{s,t}&=&y_{\varphi^{p_m^n-n}(\alpha_{p_m^n})}^{m,t}
    \\ & = &
    \mathfrak{w}_{\varphi^{p_m^n-n}(\alpha_{p_m^n})}\cdots\mathfrak{w}_{\varphi^{p_m^n-n+t-1}(\alpha_{p_m^n})}
    \underbrace{x^m_{\varphi^{p_m^n+t-n}(\alpha_{p_m^n})}}_{0} = 0
    \end{eqnarray*}
    so $y_{\varphi^{p_m^n-n}(\alpha_{p_m^n})}^{m,n}\neq y_{\varphi^{p_m^n-n}(\alpha_{p_m^n})}^{s,t}$
    and $\sigma_{\varphi,\mathfrak{w}}^n(x^m)\neq \sigma_{\varphi,\mathfrak{w}}^t(x^s)$.
\\
Using the above cases
$\{\{\sigma_{\varphi,\mathfrak{w}}^n(x^m)\}_{n\geq1}:m\geq1\}$ is
a collection of  pairwise disjoint one to one sequences, thus
$\eset(\sigma_{\varphi,\mathfrak{w}})=\mathsf{o}(\sigma_{\varphi,\mathfrak{w}})=+\infty$.
\end{proof}
\begin{corollary}\label{maryam60}
$\sigma_{\varphi,\mathfrak{w}}:F^\Gamma\to F^\Gamma$ is quasi--periodic if and only if it is pointwise quasi--periodic.
\end{corollary}
\begin{proof}
We have $\mathsf{o}(\sigma_{\varphi,\mathfrak{w}})=(\eset(\sigma_{\varphi,\mathfrak{w}})=)0$ if and only if 
$\sigma_{\varphi,\mathfrak{w}}$ is pointwise quasi--periodic, now use Theorem~\ref{ssalam10}.
\end{proof}
\section{Contravariant set--theoretical entropy of \\ finite fibre $\sigma_{\varphi,\mathfrak{w}}:F^\Gamma\to F^\Gamma$}\label{maryam70}
\noindent Since $\sigma_{\varphi,\mathfrak{w}}:F^\Gamma\to
F^\Gamma$ is an endomorphism of the abelian additive group
$(F^\Gamma,+)$, by~\cite[Theorem A]{string} we have
$\mathsf{a}(\sigma_{\varphi,\mathfrak{w}})\in\{0,+\infty\}$. In
this section we characterize and show that
$\ecset(\sigma_{\varphi,\mathfrak{w}})(=\mathsf{a}(\sigma_{\varphi,\mathfrak{w}}))$
for a finite fibre $\sigma_{\varphi,\mathfrak{w}}:F^\Gamma\to
F^\Gamma$ depends only on $\varphi$ and
$\rm{supp}(\mathfrak{w})  =\{\alpha\in\Gamma:\mathfrak{w}_\alpha\neq0\}$. In this section let:
\[\Upsilon  := \bigcup\{\varphi^{-n}(\Gamma\setminus\rm{supp}(\mathfrak{w})):n\geq0\}\:,  \:
\Lambda := \bigcap\{\varphi^{-n}({\rm supp}({\mathfrak w})):n\geq0\}= \Gamma\setminus\Upsilon\:.\]
Note that  $\varphi(\Lambda)\subseteq\Lambda$.
\begin{exam}
If $\mathfrak{w}=(1)_{\alpha\in\Gamma}$, then $\sigma_{\varphi,\mathfrak{w}}=\sigma_\varphi:F^\Gamma\to F^\Gamma$ is 
just a generalized shift moreover for this case $\rm{supp}(\mathfrak{w}) =\Lambda=\Gamma$, and $\Upsilon=\varnothing$.
\end{exam}
\begin{lemma}\label{badr10}
${\rm sc}(\sigma_{\varphi,\mathfrak{w}})\subseteq\left\{(x_\alpha)_{\alpha\in\Gamma}:\forall\beta\in\Upsilon,
\:\:x_\beta=0\right\}$.
\end{lemma}
\begin{proof}
Consider $(x_\alpha)_{\alpha\in\Gamma}\in {\rm sc}(\sigma_{\varphi,\mathfrak{w}})$ and
$\beta\in\Upsilon$, so there exists $n\geq0$ with
$\mathfrak{w}_{\varphi^n(\beta)}=0$. Since
$(x_\alpha)_{\alpha\in\Gamma}\in {\rm sc}(\sigma_{\varphi,\mathfrak{w}})
\subseteq\sigma_{\varphi,\mathfrak{w}}^{n+1}(F^\Gamma)$ there
exists $(y_\alpha)_{\alpha\in\Gamma}$ with
$\sigma_{\varphi,\mathfrak{w}}^{n+1}((y_\alpha)_{\alpha\in\Gamma})=(x_\alpha)_{\alpha\in\Gamma}$.
In particular
\[x_\beta=\mathfrak{w}_{\beta}\mathfrak{w}_{\varphi(\beta)}\cdots\underbrace{\mathfrak{w}_{\varphi^n(\beta)}}_{0}
y_{\varphi^{n+1}(\beta)}=0\]
which completes the proof.
\end{proof}
$\:$ \\
Let's call subset $M$ of $\Gamma$, \textit{$\varphi-$invariant} if $\varphi(M)\subseteq M$.
\begin{lemma}\label{badr20}
If $M$ is a $\varphi-$invariant subset  of $\Gamma$, then
$(\sigma_{\varphi,\mathfrak{w}}(x))^M=\sigma_{\varphi\restriction_M,\mathfrak{w}^M}(x^M)$
for all $x\in F^\Gamma$.
\end{lemma}
\begin{proof}
For $x=(x_\alpha)_{\alpha\in\Gamma}\in F^\Gamma$, we have:
\begin{eqnarray*}
(\sigma_{\varphi,\mathfrak{w}}(x))^M & = &
    ((\mathfrak{w}_\alpha x_{\varphi(\alpha)})_{\alpha\in\Gamma})^M
    =(\mathfrak{w}_\alpha x_{\varphi(\alpha)})_{\alpha\in M}\\
& = & (\mathfrak{w}_\alpha x_{\varphi\restriction_M(\alpha)})_{\alpha\in M}=
    \sigma_{\varphi\restriction_M,\mathfrak{w}^M}(x^M)\:.
\end{eqnarray*}
\end{proof}
\noindent Note that for $f:A\to A$, if $\{x_n\}_{n\geq1}$ is an
$f-$anti-orbit sequence, then
$\{x_n:n\geq1\}\subseteq\bigcap\{f^n(A):n\geq1\}={\rm sc}(f)$.
\begin{corollary}\label{badr25}
$\mathsf{a}(\sigma_{\varphi,\mathfrak{w}})=\mathsf{a}(\sigma_{\varphi\restriction_\Lambda,\mathfrak{w}^\Lambda})$.
\end{corollary}
\begin{proof}
If $\Lambda=\varnothing$, then by Lemma~\ref{badr10}
we have  ${\rm sc}(\sigma_{\varphi,\mathfrak{w}})=\{(0)_{\alpha\in\Gamma}\}$. So
$\mathsf{a}(\sigma_{\varphi,\mathfrak{w}})=\mathsf{a}(\sigma_{\varphi,\mathfrak{w}}\restriction_{{\rm sc}(\sigma_{\varphi,\mathfrak{w}})})
=0=
\mathsf{a}(\sigma_{\varphi\restriction_\Lambda,\mathfrak{w}^\Lambda})$.
\\
Now suppose $\Lambda\neq\varnothing$. For $m\geq1$ if
$\{x_{1,n}\}_{n\geq1},\ldots,\{x_{m,n}\}_{n\geq1}$ are $m$ one to
one $\sigma_{\varphi,\mathfrak{w}}-$anti-orbit disjoint
sequences, by Lemma~\ref{badr10} the map $\mathop{{\rm sc}(\sigma_{\varphi,\mathfrak{w}})\to
F^\Lambda}\limits_{\SP \SP \SP x\mapsto x^\Lambda}$ is one to one, then
$\{x_{1,n}^\Lambda\}_{n\geq1},\ldots,\{x_{m,n}^\Lambda\}_{n\geq1}$
are $m$ one to one disjoint sequences. On the other hand, by
Lemma~\ref{badr20},
$\{x_{1,n}^\Lambda\}_{n\geq1},\ldots,\{x_{m,n}^\Lambda\}_{n\geq1}$
are
$\sigma_{\varphi\restriction_\Lambda,\mathfrak{w}^\Lambda}-$anti-orbit
sequences. Hence
$\mathsf{a}(\sigma_{\varphi,\mathfrak{w}})\leq\mathsf{a}(\sigma_{\varphi\restriction_\Lambda,\mathfrak{w}^\Lambda})$.
\\
For each $x=(x_\alpha)_{\alpha\in\Lambda}\in F^\Lambda$, define
$\overline{x}=(\overline{x}_\alpha)_{\alpha\in\Gamma}$ with
$\overline{x}_\alpha:=0$ for all
$\alpha\in\Gamma\setminus\Lambda=\Upsilon$ and
$\overline{x}_\alpha:=x_\alpha$ for all $\alpha\in\Lambda$, so if
$\{z_{1,n}\}_{n\geq1},\ldots,\{z_{m,n}\}_{n\geq1}$ are $m$ one to
one
$\sigma_{\varphi\restriction_\Lambda,\mathfrak{w}^\Lambda}-$anti-orbit
disjoint sequences, then
$\{\overline{z}_{1,n}\}_{n\geq1},\ldots,\{\overline{z}_{m,n}\}_{n\geq1}$
are $m$ one to one disjoint sequences too. Consider
$x=(x_\alpha)_{\alpha\in\Lambda},y=(y_\alpha)_{\alpha\in\Lambda}\in
F^\Lambda$. By \ref{Join}, for each
$\alpha\in\Upsilon$ we have:
\begin{equation}\label{JoinJoin}
\begin{array}{rcl}
\alpha\in\Upsilon & \Rightarrow & \mathfrak{w}_\alpha=0\vee \varphi(\alpha)\in\Upsilon \\
& \Rightarrow & \mathfrak{w}_\alpha=0\vee \overline{x}_{\varphi(\alpha)}=0 \\
& \Rightarrow & \mathfrak{w}_\alpha  \overline{x}_{\varphi(\alpha)}=0
\end{array}
\end{equation}
therefore:
\begin{eqnarray*}
\sigma_{\varphi\restriction_\Lambda,\mathfrak{w}^\Lambda}(x)=y
    & \Rightarrow & \forall\alpha\in\Lambda\:\:\mathfrak{w}_\alpha x_{\varphi(\alpha)}=y_\alpha \\
    & \mathop{\Rightarrow}\limits^{(\varphi(\Lambda)\subseteq\Lambda)} & \forall\alpha\in\Lambda\:\:\mathfrak{w}_\alpha \overline{x}_{\varphi(\alpha)}=\overline{y}_\alpha \\
    & \mathop{\Rightarrow}\limits^{\ref{JoinJoin}} & (\forall\alpha\in\Lambda\:\:\mathfrak{w}_\alpha \overline{x}_{\varphi(\alpha)}=\overline{y}_\alpha)\wedge(\forall\alpha\in \Upsilon\:\:\mathfrak{w}_\alpha \overline{x}_{\varphi(\alpha)}=0=\overline{y}_\alpha) \\
    & \Rightarrow & \forall\alpha\in\Gamma\:\:\mathfrak{w}_\alpha \overline{x}_{\varphi(\alpha)}=\overline{y}_\alpha \\
    & \Rightarrow & \sigma_{\varphi,\mathfrak{w}}(\overline{x})=\overline{y}\:,
\end{eqnarray*}
Thus $\sigma_{\varphi,\mathfrak{w}}(\overline{x})=\overline{\sigma_{\varphi\restriction_\Lambda,\mathfrak{w}^\Lambda}(x)}$
for all $x\in F^\Lambda$, which shows
$\{\overline{z}_{1,n}\}_{n\geq1},\ldots,\{\overline{z}_{m,n}\}_{n\geq1}$
are $\sigma_{\varphi,\mathfrak{w}}-$anti-orbit sequences too.
Hence $\mathsf{a}(\sigma_{\varphi,\mathfrak{w}})\geq\mathsf{a}(\sigma_{\varphi\restriction_\Lambda,\mathfrak{w}^\Lambda})$.
\end{proof}
\begin{lemma}\label{badr28}
If $\sup(\{n:\exists\alpha\in\Gamma,\:\:(n,\alpha)\in{\mathcal
T}\}\cup\{0\})<+\infty$, then
$\mathsf{a}(\sigma_{\varphi,\mathfrak{w}})=0$. 
\end{lemma}
\begin{proof}
Suppose
$\sup(\{n:\exists\alpha\in\Gamma,\:\:(n,\alpha)\in{\mathcal
T}\}\cup\{0\})<+\infty$, then by Lemma~\ref{salam10}, there exist
$m>n\geq1$ with
$\sigma_{\varphi,\mathfrak{w}}^n=\sigma_{\varphi,\mathfrak{w}}^m$.
If $\{x_k\}_{k\geq1}$ is a
$\sigma_{\varphi,\mathfrak{w}}-$anti-orbit sequence, then
$x_m=\sigma_{\varphi,\mathfrak{w}}^n(x_{n+m})=\sigma_{\varphi,\mathfrak{w}}^m(x_{n+m})=x_n$
and $\{x_k\}_{k\geq1}$ is not one to one, thus
$\mathsf{a}(\sigma_{\varphi,\mathfrak{w}})=0$.
\end{proof}
\subsection{An equivalence relation}
For $\alpha,\beta\in\Gamma$, let $\alpha\Re\beta$ if there exists $n\geq1$ with $\varphi^n(\alpha)=\varphi^n(\beta)$.
Then $\Re$ is an equivalence relation on $\Gamma$. Note that if $\alpha\Re\beta$, then  there exists $n\geq1$ with $\varphi^n(\alpha)=\varphi^n(\beta)$ hence $\varphi^n(\varphi(\alpha))=\varphi^n(\varphi(\beta))$, therefore
$\varphi(\alpha)\Re\varphi(\beta)$. 
To the above discussion 
$\tilde{\varphi}:\mathop{\frac{\Gamma}{\Re}\to\frac{\Gamma}{\Re}}\limits_{\frac{\alpha}{\Re}\mapsto\frac{\varphi(\alpha)}{\Re}}$ 
is well--defined.
By \cite[Lemma 3.5]{nili}, $\mathsf{a}(\sigma_\varphi)=\mathsf{a}(\sigma_{\tilde{\varphi}})$, so $\Re$ and $\tilde{\varphi}$
are useful to computing contravariant set--theoretical entropy of finite fibre $\sigma_\varphi$.
\\
Let's bring some properties of $\tilde{\varphi}$ and $\Re$.
\begin{rem}
We have:
\begin{itemize}
\item[1.] $\tilde{\varphi}:\frac{\Gamma}{\Re}\to\frac{\Gamma}{\Re}$ is one to one,
\item[2.] $\mathsf{f}:{\rm sc}(\sigma_\varphi)\to F^{\frac\Gamma\Re}$ with $\mathsf{f}((x_\alpha)_{\alpha\in\Gamma})=(x_\alpha)_{\frac\alpha\Re
\in\frac\Gamma\Re}$ is (well-defined and) one to one~\cite[Note 3.2]{nili},
\item[3.] if $\sigma_\varphi:F^\Gamma\to F^\Gamma$ is of finite fibre, then 
	$\sigma_{\tilde{\varphi}}:F^{\frac{\Gamma}{\Re}}\to F^{\frac{\Gamma}{\Re}}$ is of finite fibre too~\cite[Lemma 3.4]{nili}, 
\item[4.] $\mathsf{a}(\sigma_\varphi)=\mathsf{a}(\sigma_{\tilde{\varphi}})$~\cite[Lemma 3.5]{nili}.
\end{itemize}
\end{rem}
\begin{lemma}\label{badr2000}
We have:
\begin{itemize}
\item $\{\frac{\alpha}{\Re}:\alpha\in{\rm Per}(\varphi)\}={\rm
Per}(\tilde{\varphi})$,
\item ${\rm per}(\frac{\alpha}{\Re})={\rm
per}(\alpha)$ for each $\alpha\in {\rm Per}(\varphi)$.
\end{itemize}
\end{lemma}
\begin{proof}
Consider $\theta\in{\rm Per}(\varphi)$ with $n={\rm per}(\theta)$, then
$\tilde{\varphi}^n(\frac\theta\Re)=\frac{\varphi^n(\theta)}{\Re}=\frac\theta\Re$. So
$\frac\theta\Re\in{\rm Per}(\tilde{\varphi})$ with $m:={\rm per}(\frac\theta\Re)\leq{\rm per}(\theta)$. Also
 $\frac\theta\Re=\tilde{\varphi}^m(\frac\theta\Re)=\frac{\varphi^m(\theta)}{\Re}$
and there exists $k\geq1$ with
$\varphi^k(\theta)=\varphi^{k+m}(\theta)$, thus
$\theta=\varphi^{nk}(\theta)=\varphi^{nk-k+k}(\theta)=\varphi^{nk-k+k+m}(\theta)=\varphi^{nk+m}(\theta)=\varphi^{m}(\theta)$
and ${\rm per}(\frac\theta\Re)=m\geq {\rm per}(\theta)$, which
leads to ${\rm per}(\frac\theta\Re)={\rm per}(\theta)$. Thus:
\[\{\frac{\alpha}{\Re}:\alpha\in{\rm Per}(\varphi)\}\subseteq{\rm Per}(\tilde{\varphi})\]
and
\[{\rm per}(\frac{\alpha}{\Re})={\rm per}(\alpha)\:\:(\forall\alpha\in {\rm Per}(\varphi))\:.\]
Now let $\beta\in\Gamma$ with $\frac\beta\Re\in{\rm
Per}(\tilde{\varphi})$, there exists $t\geq1$ with
$\frac\beta\Re=\tilde{\varphi}^t(\frac\beta\Re)$, thus
$\frac{\varphi^t(\beta)}{\Re}=\frac\beta\Re$ and
$\varphi^t(\beta)\Re\beta$, thus there exists $l\geq1$ with
$\varphi^{l+t}(\beta)=\varphi^l(\beta)$, hence
$\varphi^{t(l+1)}(\beta)=\varphi^{tl-l+l+t}(\beta)=\varphi^{tl-l+l}(\beta)=\varphi^{tl}(\beta)$,
therefore $\varphi^t(\beta)\in{\rm Per}(\varphi)$ and
$\frac{\varphi^t(\beta)}{\Re}=\frac\beta\Re$ which shows
$\frac\beta\Re\in\{\frac{\alpha}{\Re}:\alpha\in{\rm
Per}(\varphi)\}$ and completes the proof.
\end{proof}
\begin{lemma}\label{badr30}
The following statements are equivalent:
\begin{itemize}
\item[1.] $\sigma_{\varphi,\mathfrak{w}}:F^\Gamma\to F^\Gamma$ is of finite fibre,
\item[2.] $\Gamma\setminus\varphi(\rm{supp}(\mathfrak{w}))$ is finite,
\item[3.] There exists $N\geq1$, such that for each $x\in F^\Gamma$, we have $\vert\sigma_{\varphi,\mathfrak{w}}^{-1}(x)\vert\leq N$.
\end{itemize}
In particular, if for each $\alpha\in\Gamma$, $\mathfrak{w}_\alpha\neq 0$, then the following statements are equivalent:
\begin{itemize}
\item $\sigma_{\varphi,\mathfrak{w}}:F^\Gamma\to F^\Gamma$ is of finite fibre,
 \item $\sigma_\varphi:F^\Gamma\to F^\Gamma$ is of finite fibre,
 \item  $\Gamma\setminus \varphi(\Gamma)$ is a finite set.
\end{itemize}
\end{lemma}
\begin{proof}
Consider $x=(x_\alpha)_{\alpha\in\Gamma}\in F^\Gamma$, then:
\begin{eqnarray*}
\sigma^{-1}_{\varphi,\mathfrak w}(\sigma_{\varphi,\mathfrak w}(x)) & = &
    \{y\in F^\Gamma:\sigma_{\varphi,\mathfrak w}(y)=\sigma_{\varphi,\mathfrak w}(x)\} \\
& = & \{(y_\alpha)_{\alpha\in\Gamma}\in
F^\Gamma:\forall\alpha\in\Gamma, \SP
    \mathfrak{w}_\alpha y_{\varphi(\alpha)}=\mathfrak{w}_\alpha x_{\varphi(\alpha)}\} \\
& = & \{(y_\alpha)_{\alpha\in\Gamma}\in
F^\Gamma:\forall\alpha\in\rm{supp}(\mathfrak{w}), \SP
    y_{\varphi(\alpha)}= x_{\varphi(\alpha)}\} \\
& = & \{(y_\alpha)_{\alpha\in\Gamma}\in
F^\Gamma:\forall\alpha\in\varphi(\rm{supp}(\mathfrak{w})),
\SP
    y_\alpha= x_\alpha\}
\end{eqnarray*}
Hence $\sigma^{-1}_{\varphi,\mathfrak w}(\sigma_{\varphi,\mathfrak w}(x))$ and $F^{\Gamma\setminus
\varphi(\rm{supp}(\mathfrak{w}))}$ are equipotent. Therefore $\sigma_{\varphi,\mathfrak w}$ is of finite
fibre if and only if $\Gamma\setminus
\varphi(\rm{supp}(\mathfrak{w}))$ is finite.
\end{proof}
\begin{corollary}\label{jadid100}
If  $\sigma_{\varphi,\mathfrak{w}}:F^\Gamma\to F^\Gamma$ is of finite fibre,
then 
$\sigma_{\varphi}:F^\Gamma\to F^\Gamma$ 
and $\sigma_{\varphi\restriction_\Lambda,\mathfrak{w}^\Lambda}:F^\Lambda\to F^\Lambda$
are of finite fibre too.
\end{corollary}
\begin{proof}
Suppose $\sigma_{\varphi,\mathfrak{w}}:F^\Gamma\to F^\Gamma$ is of finite fibre, then by Lemma~\ref{badr30},
$\Gamma\setminus\varphi(\rm{supp}(\mathfrak{w}))$ is finite. Hence $\Gamma\setminus\varphi(\rm{supp}((1)_{\alpha\in\Gamma}))=\Gamma\setminus\varphi(\Gamma)
(\subseteq \Gamma\setminus\varphi(\rm{supp}(\mathfrak{w})))$ is finite too. So 
by Lemma~\ref{badr30}, $\sigma_\varphi=\sigma_{\varphi,(1)_{\alpha\in\Gamma}}:F^\Gamma\to F^\Gamma$ is of finite fibre.
\\
Moreover, for
$A:=\{(x_\alpha)_{\alpha\in\Gamma}:\forall\beta\in\Upsilon,
\:\:x_\beta=0\}$,
$\sigma_{\varphi,\mathfrak{w}}\restriction_A:A\to F^\Gamma$ is of
finite fibre too. On the other hand, for each $\alpha\in\Upsilon$
we have:
\begin{equation}\label{Join}
\begin{array}{rcl}
\alpha\in\Upsilon & \Rightarrow & \exists n\geq0\:\:\mathfrak{w}_{\varphi^n(\alpha)}=0 \\
& \Rightarrow & \mathfrak{w}_\alpha=0 \vee (\exists n\geq1,\:\:\mathfrak{w}_{\varphi^n(\alpha)}=0) \\
& \Rightarrow & \mathfrak{w}_\alpha=0 \vee \varphi(\alpha)\in\Upsilon
\end{array}
\end{equation}
So for each $(x_\alpha)_{\alpha\in\Gamma}\in A$ the following
implications are valid:
\begin{eqnarray*}
(x_\alpha)_{\alpha\in\Gamma}\in A & \Rightarrow & \forall\alpha\in\Upsilon\:\: x_\alpha=0 \\
& \mathop{\Rightarrow}\limits^{\ref{Join}} & \forall\alpha\in\Upsilon\:\: x_{\varphi(\alpha)}=0\vee\mathfrak{w}_\alpha=0 \\
& \Rightarrow & \forall\alpha\in\Upsilon\:\: \mathfrak{w}_\alpha x_{\varphi(\alpha)}=0 \\
& \Rightarrow & \sigma_{\varphi,\mathfrak{w}}((x_\alpha)_{\alpha\in\Gamma})\in A\:.
\end{eqnarray*}
Thus $\sigma_{\varphi,\mathfrak{w}}(A)\subseteq A$, and $\sigma_{\varphi,\mathfrak{w}}\restriction_A:A\to A$
is of finite fibre. Since $\mathop{k:A\to F^\Lambda}\limits_{x\mapsto x^\Lambda}$ is bijective
and $\sigma_{\varphi\restriction_\Lambda,\mathfrak{w}^\Lambda}=k\circ\sigma_{\varphi,\mathfrak{w}}\restriction_A\circ k^{-1}$, we have the desired result.
\end{proof}
 In the above corollary stemmed from Lemma~\ref{badr30}, if $\sigma_{\varphi,\mathfrak{w}}:F^\Gamma\to F^\Gamma$ is of finite fibre,
then $\sigma_{\varphi}:F^\Gamma\to F^\Gamma$ is of finite fibre too, the following counterexample shows that the reversed
the implication is not true.
\begin{counterexample}
Let $\Gamma=\mathbb{Z}$,
$\varphi:\mathop{\mathbb{Z}\to\mathbb{Z}}\limits_{n\mapsto n+1}$,
$\mathfrak{w}_{2n}=0$ and $\mathfrak{w}_{2n+1}=1$ for
$n\in\mathbb Z$. Then $\varphi(\Gamma)=\Gamma$ and
$\sigma_{\varphi}:F^\Gamma\to F^\Gamma$ is of finite fibre while
$\rm{supp}(\mathfrak{w})=\Gamma\setminus\varphi(\rm{supp}(\mathfrak{w}))=2\mathbb{Z}+1$ 
is infinite, hence $\sigma_{\varphi,\mathfrak{w}}:F^\Gamma\to
F^\Gamma$ is not of finite fibre.
\end{counterexample}
\subsection{Towards computing $\ecset(\sigma_{\varphi,\mathfrak{w}})$}\label{maryam30}
By Lemma~\ref{badr30}, $\sigma_{\varphi,\mathfrak{w}}:F^\Gamma\to F^\Gamma$ is of finite fibre
if and only if $\Gamma\setminus\varphi(\rm{supp}(\mathfrak{w}))$ is finite.
\begin{convention}\label{maryam20}
\underline{In this subsection suppose} $\sigma_{\varphi,\mathfrak{w}}:F^\Gamma\to F^\Gamma$ is of finite fibre, i.e.
$\Gamma\setminus\varphi(\rm{supp}(\mathfrak{w}))$ is finite.
\end{convention}
Finite fibreness of $\sigma_{\varphi,\mathfrak{w}}:F^\Gamma\to F^\Gamma$ leads us to the following corollaries.
\begin{corollary}\label{badazbadr25}
$\ecset(\sigma_{\varphi,\mathfrak{w}})=\ecset(\sigma_{\varphi\restriction_\Lambda,\mathfrak{w}^\Lambda})$.
\end{corollary}
\begin{proof}
Use Corollary \ref{badr25}.
\end{proof}
\begin{corollary}\label{badr27}
If $\sup(\{n:\exists\alpha\in\Gamma,\:\:(n,\alpha)\in{\mathcal
T}\}\cup\{0\})<+\infty$, then
$\ecset(\sigma_{\varphi,\mathfrak{w}})=0$.
\end{corollary}
\begin{proof}
Use Lemma~\ref{badr28}.
\end{proof}
\noindent In the following  propositions, we restrict ourselves to conditions which make us closer to $\sigma_\varphi$'s situation. Also we will use $\Re$ and $\tilde{\varphi}$ in the 
proof of the following theorem.
\begin{theorem}\label{badr40}
Suppose ${\rm supp}(\mathfrak{w})=\Gamma$ and at least one of the following conditions occurs:
\begin{itemize}
\item $\varphi$ has a non--quasi--periodic point $\alpha\in\Gamma$,
\item ${\rm Per}(\varphi)\neq\varnothing$ and $\sup\{{\rm per}(\alpha):\alpha\in{\rm Per}(\varphi)\}=+\infty$,
\end{itemize}
then $\ecset(\sigma_{\varphi,\mathfrak{w}})=\ecset(\sigma_\varphi)=+\infty$.
\end{theorem}
\begin{proof}
For $r\in F$, let:
\begin{equation}\label{kheybar300}
r^*:=\left\{\begin{array}{lc} 0, &  r=0\:, \\ 1, & 
 r\neq0\:.\end{array}\right.
\end{equation}
Consider $h:F^\Gamma\to\{0,1\}^\Gamma$ with
$h((r_\alpha)_{\alpha\in\Gamma})=(r^*_\alpha)_{\alpha\in\Gamma}$
($(r_\alpha)_{\alpha\in\Gamma}\in F^\Gamma$). For each
$(x_\alpha)_{\alpha\in\Gamma}\in F^\Gamma$, we have:
\begin{eqnarray*}
h(\sigma_{\varphi,\mathfrak{w}}((x_\alpha)_{\alpha\in\Gamma})) & = &
    h((\mathfrak{w}_\alpha x_{\varphi(\alpha)})_{\alpha\in\Gamma}) \\
& = & ((\mathfrak{w}_\alpha x_{\varphi(\alpha)})^*)_{\alpha\in\Gamma}=
    (\mathfrak{w}^*_\alpha x^*_{\varphi(\alpha)})_{\alpha\in\Gamma} \\
& = & (1 x^*_{\varphi(\alpha)})_{\alpha\in\Gamma}=(x^*_{\varphi(\alpha)})_{\alpha\in\Gamma} \\
& = & \sigma_\varphi((x^*_\alpha)_{\alpha\in\Gamma})=\sigma_\varphi(h((x_\alpha)_{\alpha\in\Gamma}))
\end{eqnarray*}
Hence $h\circ\sigma_{\varphi,\mathfrak{w}}=\sigma_\varphi\circ h$ and the following diagram commutes:
\[\xymatrix{F^\Gamma\ar[r]^{\sigma_{\varphi,\mathfrak{w}}}\ar[d]_h & F^\Gamma\ar[d]^h \\
\{0,1\}^\Gamma\ar[r]^{\sigma_\varphi} & \{0,1\}^\Gamma }\] It's
evident that $h:F^\Gamma\to\{0,1\}^\Gamma$ is surjective,
moreover, by Corollary~\ref{jadid100}, \linebreak
$\sigma_\varphi:\{0,1\}^\Gamma\to\{0,1\}^\Gamma$ is of finite fibre.
By ~\cite[Lemma 3.2.22 (b)]{anna} we have:
\begin{equation}\label{+}
\ecset(\sigma_\varphi)\leq\ecset(\sigma_{\varphi,\mathfrak{w}})\:.
\end{equation}
By Lemma~\ref{badr2000},  at least one of the following
conditions occurs:
\\
$\bullet$ $\tilde{\varphi}$ has a non--quasi periodic point in $\frac{\Gamma}{\Re}$,
\\
$\bullet$ ${\rm Per}(\tilde{\varphi})\neq\varnothing$ and
    $\sup\{{\rm per}(D):D\in{\rm Per}(\tilde{\varphi})\}=\sup\{{\rm per}(\alpha):\alpha\in{\rm Per}(\varphi)\}=+\infty$,
\\
then by~\cite[Corollary 3.9]{nili}
\begin{equation}\label{++}
\ecset(\sigma_\varphi)=+\infty\:.
\end{equation}
Using \ref{+} and \ref{++} we have
$\ecset(\sigma_{\varphi,\mathfrak{w}})=\ecset(\sigma_\varphi)=+\infty$.
\end{proof}
\begin{lemma}\label{badr50}
Suppose  
${\rm Per}(\varphi)={\rm Fix}(\varphi)$, 
and ${\rm QPer}(\varphi)={\rm supp}(\mathfrak{w})=\Gamma$,   
then:
\begin{itemize}
\item[1.] $\sigma_{\varphi\restriction_{{\rm Fix}(\varphi)},\mathfrak{w}^{{\rm Fix}(\varphi)}}:
    F^{{\rm Fix}(\varphi)}\to F^{{\rm Fix}(\varphi)}$ is bijective,
\item[2.] $\mathfrak{p}:{\rm sc}(\sigma_{\varphi,\mathfrak{w}})\to F^{{\rm Fix}(\varphi)}$ with $\mathfrak{p}(x)=x^{{\rm Fix}(\varphi)}$, is bijective,
\item[3.] $\ecset(\sigma_{\varphi,\mathfrak{w}})=\ecset(\sigma_{\varphi\restriction_{{\rm Fix}(\varphi)},\mathfrak{w}^{{\rm Fix}(\varphi)}})=0$.
\end{itemize}
\end{lemma}
\begin{proof} Using ${\rm QPer}(\varphi)=\Gamma$ we have ${\rm Per}(\varphi)\neq\varnothing$. therefor
${\rm Fix}(\varphi)\neq\varnothing$, since ${\rm Per}(\varphi)={\rm Fix}(\varphi)$.
\\
{\bf 1)} Note that $\varphi\restriction_{{\rm
Fix}(\varphi)}=id_{{\rm Fix}(\varphi)}$, moreover, for each
$(x_\alpha)_{\alpha\in{{\rm
Fix}(\varphi)}},(y_\alpha)_{\alpha\in{{\rm Fix}(\varphi)}}\in
F^{{\rm Fix}(\varphi)}$, we have:
\[\sigma_{\varphi\restriction_{{\rm Fix}(\varphi)},\mathfrak{w}^{{\rm Fix}(\varphi)}}((\mathfrak{w}_\alpha^{-1}y_\alpha)_{\alpha\in{{\rm Fix}(\varphi)}})
=(y_\alpha)_{\alpha\in{{\rm Fix}(\varphi)}}\]
and
\begin{eqnarray*}
\sigma_{\varphi\restriction_{{\rm Fix}(\varphi)},\mathfrak{w}^{{\rm Fix}(\varphi)}}((x_\alpha)_{\alpha\in{{\rm Fix}(\varphi)}})
& = &
\sigma_{\varphi\restriction_{{\rm Fix}(\varphi)},\mathfrak{w}^{{\rm Fix}(\varphi)}}((y_\alpha)_{\alpha\in{{\rm Fix}(\varphi)}})
    \\
& \Rightarrow &
(\mathfrak{w}_\alpha x_\alpha)_{\alpha\in{{\rm Fix}(\varphi)}}
=(\mathfrak{w}_\alpha y_\alpha)_{\alpha\in{{\rm Fix}(\varphi)}}
    \\
& \Rightarrow &
\forall{\alpha\in{{\rm Fix}(\varphi)}}\SP\mathfrak{w}_\alpha x_\alpha
=\mathfrak{w}_\alpha y_\alpha
    \\
& \Rightarrow &
\forall{\alpha\in{{\rm Fix}(\varphi)}}\SP\mathfrak{w}_\alpha^{-1}\mathfrak{w}_\alpha x_\alpha
=\mathfrak{w}_\alpha^{-1}\mathfrak{w}_\alpha y_\alpha
    \\
& \Rightarrow &
\forall{\alpha\in{{\rm Fix}(\varphi)}}\SP x_\alpha=y_\alpha
    \\
& \Rightarrow & (x_\alpha)_{\alpha\in{{\rm Fix}(\varphi)}}=(y_\alpha)_{\alpha\in{{\rm Fix}(\varphi)}}
\end{eqnarray*}
hence $\sigma_{\varphi\restriction_{{\rm Fix}(\varphi)},\mathfrak{w}^{{\rm Fix}(\varphi)}}:
    F^{{\rm Fix}(\varphi)}\to F^{{\rm Fix}(\varphi)}$ is bijective.
\\
{\bf 2)} We prove this item via the following claims:
\\
{\it Claim I.} For $\beta\in\Gamma$, $q\geq2$ with
$\varphi^q(\beta)\in{\rm Fix}(\varphi)$ and
$x=(x_\alpha)_{\alpha\in\Gamma}\in {\rm
sc}(\sigma_{\varphi,\mathfrak{w}})$, we have
$x_\beta=\mathfrak{w}_{\varphi^q(\beta)}^{-q}\mathfrak{w}_\beta\mathfrak{w}_{\varphi(\beta)}\cdots
    \mathfrak{w}_{\varphi^{q-1}(\beta)} x_{\varphi^q(\beta)}$.
\\
{\it Proof of Claim I.} Consider $x=(x_\alpha)_{\alpha\in\Gamma}\in {\rm
sc}(\sigma_{\varphi,\mathfrak{w}})$. For $\beta\in\Gamma$, there
exists $q\geq2$ such that $\varphi^q(\beta)\in {\rm
Fix}(\varphi)$. There exists $z=(z_\alpha)_{\alpha\in\Gamma}$ with
$\sigma_{\varphi,\mathfrak{w}}^q(z)=x$.
\begin{eqnarray*}
\sigma_{\varphi,\mathfrak{w}}^q(z)=x & \Rightarrow & \forall\alpha\in\Gamma\:\:
    \mathfrak{w}_\alpha\cdots\mathfrak{w}_{\varphi^{q-1}(\alpha)}z_{\varphi^q(\alpha)}=x_\alpha \\
& \Rightarrow & \forall s\geq0  \:\:
    \mathfrak{w}_{\varphi^s(\beta)}\cdots\mathfrak{w}_{\varphi^{s+q-1}(\beta)}z_{\varphi^{s+q}
    (\beta)}=x_{\varphi^s(\beta)} \\
& \Rightarrow & \forall s\geq0  \:\:
    \mathfrak{w}_{\varphi^s(\beta)}\cdots\mathfrak{w}_{\varphi^{s+q-1}(\beta)}z_{\varphi^{q}
    (\beta)}=x_{\varphi^s(\beta)} \\
& \Rightarrow & \forall s\geq0  \:\:
    \mathfrak{w}_{\varphi^s(\beta)}x_{\varphi^{s+1}(\beta)}=x_{\varphi^s(\beta)}\mathfrak{w}_{\varphi^{s+q}(\beta)} \\
& \Rightarrow & \forall s\geq0  \:\:
    \mathfrak{w}_{\varphi^s(\beta)}x_{\varphi^{s+1}(\beta)}=x_{\varphi^s(\beta)}\mathfrak{w}_{\varphi^{q}(\beta)}
\end{eqnarray*}
Therefore:
\begin{eqnarray*}
x_\beta & = &\mathfrak{w}_{\varphi^q(\beta)}^{-1}\mathfrak{w}_\beta x_{\varphi(\beta)} \\
& = & \mathfrak{w}_{\varphi^q(\beta)}^{-2}\mathfrak{w}_\beta\mathfrak{w}_{\varphi(\beta)} x_{\varphi^2(\beta)} \\
& = & \mathfrak{w}_{\varphi^q(\beta)}^{-3}\mathfrak{w}_\beta\mathfrak{w}_{\varphi(\beta)}
    \mathfrak{w}_{\varphi^2(\beta)} x_{\varphi^3(\beta)} \\
&& \vdots \\
& = & \mathfrak{w}_{\varphi^q(\beta)}^{-q}\mathfrak{w}_\beta\mathfrak{w}_{\varphi(\beta)}\cdots
    \mathfrak{w}_{\varphi^{q-1}(\beta)} x_{\varphi^q(\beta)}
\end{eqnarray*}
\\
{\it Claim II.} $\mathfrak{p}:{\rm sc}(\sigma_{\varphi,\mathfrak{w}})\to F^{{\rm Fix}(\varphi)}$ is one to one.
\\
{\it Proof of Claim II.} Consider
$x=(x_\alpha)_{\alpha\in\Gamma},y=(y_\alpha)_{\alpha\in\Gamma}\in
{\rm sc}(\sigma_{\varphi,\mathfrak{w}})$ with
$\mathfrak{p}(x)=\mathfrak{p}(y)$. Choose $\beta\in\Gamma$, there
exists $q\geq2$ with $\varphi^q(\beta)\in{\rm Fix}(\varphi)$, now
we have:
\\
{\small $\mathfrak{p}(x)=\mathfrak{p}(y) $
\begin{eqnarray*}
& \Rightarrow & \forall\alpha\in{\rm Fix}(\varphi),\:\: x_\alpha=y_\alpha \\
& \mathop{\Rightarrow}\limits^{\varphi^q(\beta)\in{\rm Fix}(\varphi)} & x_{\varphi^q(\beta)}=y_{\varphi^q(\beta)} \\
& \Rightarrow & \mathfrak{w}_{\varphi^q(\beta)}^{-q}\mathfrak{w}_\beta\mathfrak{w}_{\varphi(\beta)}\cdots
    \mathfrak{w}_{\varphi^{q-1}(\beta)} x_{\varphi^q(\beta)}=
    \mathfrak{w}_{\varphi^q(\beta)}^{-q}\mathfrak{w}_\beta\mathfrak{w}_{\varphi(\beta)}\cdots
    \mathfrak{w}_{\varphi^{q-1}(\beta)} y_{\varphi^q(\beta)} \\
& \mathop{\Rightarrow}\limits^{({\rm Claim \: I})} & x_\beta=y_\beta
\end{eqnarray*}}
Since $x_\beta=y_\beta$ for all $\beta\in\Gamma$, we have $x=y$ and $\mathfrak{p}$ is one to one.
\\
{\it Claim III.}
For all $x=(x_\alpha)_{\alpha\in{\rm Fix}(\varphi)},y=(y_\alpha)_{\alpha\in{\rm Fix}(\varphi)}\in F^{{\rm Fix}(\varphi)}$ and $n\geq1$ we have $\sigma_{\varphi\restriction_{{\rm Fix}(\varphi)},\mathfrak{w}^{{\rm Fix}(\varphi)}}^n(y)=x$
if and only if $\mathfrak{w}_\theta^{n}y_\theta=x_\theta$ for all $\theta\in{\rm Fix}(\varphi)$.
\\
{\it Proof of Claim III.} We have:
\begin{eqnarray*}
\sigma_{\varphi\restriction_{{\rm Fix}(\varphi)},\mathfrak{w}^{{\rm Fix}(\varphi)}}^n(y)=x & \Rightarrow &
    \forall\theta\in{\rm Fix}(\varphi)\:\:
    \mathfrak{w}_\theta\mathfrak{w}_{\varphi(\theta)}\cdots\mathfrak{w}_{\varphi^{n-1}(\theta)}y_{\varphi^n(\theta)}=x_\theta \\
& \Rightarrow &     \forall\theta\in{\rm Fix}(\varphi)\:\:\mathfrak{w}_\theta^{n}y_\theta=x_\theta
\end{eqnarray*}
\\
{\it Claim IV.} $\mathfrak{p}:{\rm
sc}(\sigma_{\varphi,\mathfrak{w}})\to F^{{\rm Fix}(\varphi)}$ is
surjective. 
\\
{\it Proof of Claim IV.} Choose $x=(x_\alpha)_{\alpha\in{\rm Fix}(\varphi)}\in
F^{{\rm Fix}(\varphi)}$. For $n\geq1$, there exists
$y=(y_\alpha)_{\alpha\in{\rm Fix}(\varphi)}\in F^{{\rm
Fix}(\varphi)}$ with $\sigma_{\varphi\restriction_{{\rm
Fix}(\varphi)},\mathfrak{w}^{{\rm Fix}(\varphi)}}^n(y)=x$. For
each $\beta\in\Gamma$, choose $q_\beta\geq2$ such that
$\varphi^{q_\beta}(\beta)\in {\rm Fix}(\varphi)$. Let:
\[ A:=(\mathfrak{w}_{\varphi^{q_\alpha}(\alpha)}^{-q_\alpha}\mathfrak{w}_\alpha\mathfrak{w}_{\varphi(\alpha)}\cdots
    \mathfrak{w}_{\varphi^{q_\alpha-1}(\alpha)} x_{\varphi^{q_\alpha}(\alpha)})_{\alpha\in\Gamma}\:,\]
\[B:=(\mathfrak{w}_{\varphi^{q_\alpha}(\alpha)}^{-q_\alpha}\mathfrak{w}_\alpha\mathfrak{w}_{\varphi(\alpha)}\cdots
    \mathfrak{w}_{\varphi^{q_\alpha-1}(\alpha)} y_{\varphi^{q_\alpha}(\alpha)})_{\alpha\in\Gamma}\:.\]
Note that for each $\alpha\in\Gamma$ and $m\geq0$ we have $\varphi^{q_\alpha}(\alpha)=\varphi^{q_\alpha+m}(\alpha)$ and:
\begin{equation}\label{divideontimes}
\mathfrak{w}_{\varphi^{q_\alpha+m}(\alpha)}^{-q_\alpha-m}\mathfrak{w}_\alpha\mathfrak{w}_{\varphi(\alpha)}\cdots
    \mathfrak{w}_{\varphi^{q_\alpha+m-1}(\alpha)}=\mathfrak{w}_{\varphi^{q_\alpha}(\alpha)}^{-q_\alpha}\mathfrak{w}_\alpha\mathfrak{w}_{\varphi(\alpha)}\cdots
    \mathfrak{w}_{\varphi^{q_\alpha-1}(\alpha)}\:.
\end{equation}
So $A$ and $B$ don't depend on the way of choosing $q_\alpha$s.
Moreover, for each $\alpha\in\Gamma$, 
we have (using
\ref{divideontimes} we may suppose
$q_\alpha=q_{\varphi^n(\alpha)}=:q$):

$\alpha$th coordinate of $\sigma_{\varphi,\mathfrak{w}}^n(B)=  \mathfrak{w}_\alpha\mathfrak{w}_{\varphi(\alpha)}\cdots\mathfrak{w}_{\varphi^{n-1}(\alpha)}(\varphi^n(\alpha)th
    {\rm \: coordinate \: of \:}B)
$

$ =  \bigg(\mathop{\prod}\limits_{0\leq i\leq n-1}\mathfrak{w}_{\varphi^i(\alpha)} \bigg)
    \mathfrak{w}_{\varphi^{q_{\varphi^n(\alpha)}}(\varphi^n(\alpha))}^{-q_{\varphi^n(\alpha)}}
    \bigg(\mathop{\prod}\limits_{0\leq i\leq q_{\varphi^n(\alpha)}-1}\mathfrak{w}_{\varphi^i(\varphi^n(\alpha))}\bigg)
    y_{\varphi^{q_{\varphi^n(\alpha)}}(\varphi^n(\alpha))} 
$
    \begin{eqnarray*}
& =& \bigg(\mathop{\prod}\limits_{0\leq i\leq n-1}\mathfrak{w}_{\varphi^i(\alpha)}\bigg)
    \mathfrak{w}_{\varphi^{q}(\varphi^n(\alpha))}^{-q}
    \bigg(\mathop{\prod}\limits_{0\leq i\leq q-1}\mathfrak{w}_{\varphi^i(\varphi^n(\alpha))}\bigg)
    y_{\varphi^{q}(\varphi^n(\alpha))} \\ && \\
& = & \bigg ( \mathop{\prod}\limits_{0\leq i\leq q+n-1}\mathfrak{w}_{\varphi^i(\alpha)}\bigg)
    \mathfrak{w}_{\varphi^{q}(\varphi^n(\alpha))}^{-q}y_{\varphi^{q}(\varphi^n(\alpha))} \\ && \\ 
& \mathop{=}\limits^{\varphi^q(\alpha)\in{\rm Fix(\varphi)}} & \bigg(\mathop{\prod}\limits_{0\leq i\leq q+n-1}\mathfrak{w}_{\varphi^i(\alpha)}\bigg)
    \mathfrak{w}_{\varphi^{q}(\alpha)}^{-q}y_{\varphi^{q}(\alpha)} \\ && \\ 
& \mathop{=}\limits^{\forall i\geq q\:\varphi^i(\alpha)=\varphi^q(\alpha)} &  \bigg(\mathop{\prod}\limits_{0\leq i\leq q-1}\mathfrak{w}_{\varphi^i(\alpha)}\bigg)
    \mathfrak{w}_{\varphi^{q}(\alpha)}^{-q+n}y_{\varphi^{q}(\alpha)} \\ && \\
& \mathop{=}\limits^{({\rm Claim \: III})} & \bigg(\mathop{\prod}\limits_{0\leq i\leq q-1}\mathfrak{w}_{\varphi^i(\alpha)}\bigg)
    \mathfrak{w}_{\varphi^{q}(\alpha)}^{-q}x_{\varphi^{q}(\alpha)} 
= \alpha{\rm th \: coordinate \: of \:} A
\end{eqnarray*}
Thus
$A=\sigma_{\varphi,\mathfrak{w}}^n(B)\in\sigma_{\varphi,\mathfrak{w}}^n(F^\Gamma)$
which leads to
$A\in\bigcap\{\sigma_{\varphi,\mathfrak{w}}^n(F^\Gamma):n\geq1\}={\rm
sc}(\sigma_{\varphi,\mathfrak{w}})$, on the other hand,
$\mathfrak{p}(A)=x$ and $\mathfrak{p}:{\rm
sc}(\sigma_{\varphi,\mathfrak{w}})\to F^{{\rm Fix}(\varphi)}$ is
surjective.
\\
{\bf 3)} Consider the following commutative diagram:
\[\xymatrix{{\rm sc}(\sigma_{\varphi,\mathfrak{w}})\ar[rrr]^{\sigma_{\varphi,\mathfrak{w}}\restriction_{{\rm sc}(\sigma_{\varphi,\mathfrak{w}})}}\ar[d]_p &&& {\rm sc}(\sigma_{\varphi,\mathfrak{w}})\ar[d]^p \\
F^{{\rm Fix}(\varphi)}\ar[rrr]^{\sigma_{\varphi\restriction_{{\rm Fix}(\varphi)},\mathfrak{w}^{{\rm Fix}(\varphi)}}} &&& F^{{\rm Fix}(\varphi)}}\]
So $\sigma_{\varphi,\mathfrak{w}}\restriction_{{\rm sc}(\sigma_{\varphi,\mathfrak{w}})}=\mathfrak{p}^{-1}\circ\sigma_{\varphi\restriction_{{\rm Fix}(\varphi)},\mathfrak{w}^{{\rm Fix}(\varphi)}}\circ\mathfrak{p}$. Hence (note that for finite fibre $f:X\to X$, a sequence is an $f-$anti--orbit if and only if it is an
$f\restriction_{{\rm sc}(f)}-$anti--orbit, also use \cite[Lemma 3.2.22 (c)]{anna}):
\begin{eqnarray*}
\ecset(\sigma_{\varphi,\mathfrak{w}}) & = &
\mathsf{a}(\sigma_{\varphi,\mathfrak{w}}\restriction_{{\rm sc}(\sigma_{\varphi,\mathfrak{w}})})
 =  \mathsf{a}(\mathfrak{p}^{-1}\circ\sigma_{\varphi\restriction_{{\rm Fix}(\varphi)},\mathfrak{w}^{{\rm Fix}(\varphi)}}\circ\mathfrak{p}) \\
 & = & \ecset(\mathfrak{p}^{-1}\circ\sigma_{\varphi\restriction_{{\rm Fix}(\varphi)},\mathfrak{w}^{{\rm Fix}(\varphi)}}\circ\mathfrak{p})
=  \ecset(\sigma_{\varphi\restriction_{{\rm Fix}(\varphi)},\mathfrak{w}^{{\rm Fix}(\varphi)}})\:.
\end{eqnarray*}
Moreover, by Lemma~\ref{badr28},
$\mathsf{a}(\sigma_{\varphi\restriction_{{\rm
Fix}(\varphi)},\mathfrak{w}^{{\rm Fix}(\varphi)}})=0$, which
completes the proof.
\end{proof}
\begin{lemma}\label{badr60}
If ${\rm supp}(\mathfrak{w})=\Gamma$, 
${\rm QPer}(\varphi)=\Gamma$ and $\sup\{{\rm per}(\alpha):\alpha\in{\rm Per}(\varphi)\}<+\infty$,
then
$\ecset(\sigma_{\varphi,\mathfrak{w}})=\ecset(\sigma_\varphi)=0$.
\end{lemma}
\begin{proof}
By Corollary~\ref{jadid100},
$\sigma_\varphi:F^\Gamma\to F^\Gamma$ is of finite fibre  too.
Let $\mathfrak{u}=(\mathfrak{u}_\alpha)_{\alpha\in\Gamma}=(1)_{\alpha\in\Gamma}$.
Suppose
$\sup\{{\rm per}(\alpha):\alpha\in{\rm Per}(\varphi)\}=N<+\infty$. All periodic points of
$\varphi^N$ are fixed points of $\varphi^N$ and all points of $\Gamma$ are quasi--periodic points of $\varphi^N$,
by Lemma~\ref{badr50}:
\[\ecset(\sigma_{\varphi^N,\mathfrak{u}})= \ecset(\sigma_{\varphi^N,\mathfrak{w}\sigma_\varphi(\mathfrak{w})\cdots\sigma_{\varphi^{N-1}}(\mathfrak{w})})=0\:.\]
i.e. $\ecset(\sigma_{\varphi,\mathfrak{u}}^N)=
\ecset(\sigma_{\varphi,\mathfrak{w}}^N)=0$, which leads to (use
\cite[Proposition 3.2.40 (Logarithmic Law)]{anna}):
\[N\ecset(\sigma_{\varphi,\mathfrak{u}})= N\ecset(\sigma_{\varphi,\mathfrak{w}})=0\]
and
\[\ecset(\sigma_{\varphi,\mathfrak{u}})= \ecset(\sigma_{\varphi,\mathfrak{w}})=0\:.\]
Use the above cases and $\sigma_{\varphi,\mathfrak{u}}=\sigma_\varphi$ to complete the proof.
\end{proof}
\begin{corollary}
By Theorem~\ref{badr40} and Lemma~\ref{badr60}, if ${\rm supp}(\mathfrak{w})=\Gamma$, 
then
$\ecset(\sigma_{\varphi,\mathfrak{w}})=\ecset(\sigma_\varphi)(\in\{0,+\infty\})$.
\end{corollary}
\begin{theorem}\label{badr3000}
We have:
\begin{eqnarray*}
\ecset(\sigma_{\varphi,\mathfrak{w}}) & = & \ecset(\sigma_{\varphi\restriction_\Lambda,\mathfrak{w}^\Lambda})
 =  \ecset(\sigma_{\varphi\restriction_\Lambda}) \\
& = & \left\{\begin{array}{lc} +\infty, & \varphi {\rm \: has \: a \: non-quasi-periodic \: point \: in\:}\Lambda\:, \\
    +\infty, & {\rm Per}(\varphi)\cap\Lambda\neq\varnothing\wedge\sup\{{\rm per}(\alpha):
    \alpha\in {\rm Per}(\varphi)\cap\Lambda\}=+\infty\:, \\
    0, & {\rm otherwise}\:.
    \end{array}\right.
\end{eqnarray*}
\end{theorem}
\begin{proof}
By Corollary~\ref{badazbadr25}, we have
$\ecset(\sigma_{\varphi,\mathfrak{w}})=\ecset(\sigma_{\varphi\restriction_\Lambda,\mathfrak{w}^\Lambda})$.
We have the following cases (note that for all $\alpha\in\Lambda$, ${\mathfrak w}_\alpha\neq 0$):
\begin{itemize}
\item[1.] if ${\rm QPer}(\varphi\restriction_\Lambda)\neq\Lambda$, i.e. $\varphi$ has a non--quasi--periodic point in $\Lambda$, then 
by Theorem~\ref{badr40},
$\ecset(\sigma_{\varphi\restriction_\Lambda,\mathfrak{w}^\Lambda})=+\infty$,
\item[2.] if ${\rm Per}(\varphi\restriction_\Lambda)\neq\varnothing$ and $\sup\{{\rm per}(\alpha):\alpha\in{\rm Per}(\varphi\restriction_\Lambda)\}=+\infty$, then 
by Theorem~\ref{badr40},
$\ecset(\sigma_{\varphi\restriction_\Lambda,\mathfrak{w}^\Lambda})=+\infty$
(note that ${\rm Per}(\varphi\restriction_\Lambda)={\rm Per}(\varphi)\cap\Lambda$),
\item[3.] if ${\rm  QPer}(\varphi\restriction_\Lambda)=\Lambda$ and $\sup\{{\rm per}(\alpha):\alpha\in{\rm Per}(\varphi\restriction_\Lambda)\}<+\infty$, i.e. neither (1) occurs nor (2) occurs, then 
by Lemma~\ref{badr60},
$\ecset(\sigma_{\varphi\restriction_\Lambda,\mathfrak{w}^\Lambda})=0$.
\end{itemize}
\end{proof}
\section{Interaction between possible set--theoretical entropies of $\sigma_{\varphi,\mathfrak{w}}$}
\noindent In this section, we try to find out the interaction between possible set--theoretical
entropies arised from generalized and weighted generalized shifts. In this section, we try once more the above note via 
two corollaries and then a table. 
\begin{corollary}\label{badr6000}
Suppose
$\eset(\sigma_{\varphi})=+\infty$ and $\eset(\sigma_{\varphi,\mathfrak{w}})=0$,
then $\sigma_{\varphi,\mathfrak{w}}:F^\Gamma\to F^\Gamma$ is not finite fibre.
\end{corollary}
\begin{proof}
Suppose $\eset(\sigma_{\varphi})=+\infty$ and
$\eset(\sigma_{\varphi,\mathfrak{w}})=0$. By Theorem~\ref{ssalam10},
there exist $p>q\geq1$, such that, for each $\alpha\in\Gamma$, we
have
\[\varphi^p(\alpha)=\varphi^q(\alpha)\vee
\mathfrak{w}_\alpha\mathfrak{w}_{\varphi(\alpha)}\cdots\mathfrak{w}_{\varphi^q(\alpha)}=0\:.\]
Let $s:=p-q$, then for each $\alpha\in\Gamma$ with
$\varphi^p(\alpha)=\varphi^q(\alpha)$, we have:
\begin{eqnarray*}
\varphi^p(\alpha)=\varphi^q(\alpha) & \Rightarrow & \varphi^s(\varphi^q(\alpha))=\varphi^q(\alpha) \\
& \Rightarrow & \varphi^s(\varphi^{qs}(\alpha))=\varphi^{qs-q}(\varphi^s(\varphi^q(\alpha)))
    =\varphi^{qs-q}(\varphi^q(\alpha))=\varphi^{qs}(\alpha) \\
& \Rightarrow & \varphi^{2qs}(\alpha)=\varphi^{qs}(\varphi^{qs}(\alpha))=\varphi^{qs}(\alpha)
\end{eqnarray*}
hence for $N:=2qs$ we have:
\begin{equation}\label{boxplus}
\forall\alpha\in\Gamma\:\:\: (\varphi^{2N}(\alpha)=\varphi^N(\alpha)\vee
\mathfrak{w}_\alpha\mathfrak{w}_{\varphi(\alpha)}\cdots\mathfrak{w}_{\varphi^{N-1}(\alpha)}=0)\:.
\end{equation}
Let $\eta:=\varphi^N,\mathfrak{v}:=\mathfrak{w}\sigma_\varphi({\mathfrak w})\cdots\sigma_\varphi^{N-1}({\mathfrak w})$,
then:
\begin{itemize}
\item $\sigma_{\eta,\mathfrak{v}}=\sigma_{\varphi,\mathfrak{w}}^N$  and $\sigma_\eta=\sigma_\varphi^N$,
\item $\eset(\sigma_{\eta,\mathfrak{v}})=\eset(\sigma_{\varphi,\mathfrak{w}}^N)=N\eset(\sigma_{\varphi,\mathfrak{w}})=0$,
\item $\eset(\sigma_\eta)=\eset(\sigma_{\varphi}^N)=N\eset(\sigma_{\varphi})=+\infty$.
\end{itemize}
Let $Z=\Gamma\setminus\rm{supp}(\mathfrak{v})$, then by \ref{boxplus} we have:
\begin{equation}\label{boxplusboxplus}
\forall\alpha\in\Gamma\:\:(\mathfrak{v}_\alpha=0\vee\eta(\alpha)\in{\rm Fix}(\eta))\:,
\end{equation}
i.e., $\Gamma=Z\cup\eta^{-1}({\rm Fix}(\eta))$. Therefore:
\begin{eqnarray*}
\Gamma\setminus\eta(\rm{supp}(\mathfrak{v})) & = & (Z\cup\eta^{-1}({\rm Fix}(\eta)))\setminus\eta((Z\cup\eta^{-1}({\rm Fix}(\eta)))\setminus Z) \\
& = & (Z\cup\eta^{-1}({\rm Fix}(\eta)))\setminus\eta(\eta^{-1}({\rm Fix}(\eta))\setminus Z) \\
& \supseteq & (Z\cup\eta^{-1}({\rm Fix}(\eta)))\setminus\eta(\eta^{-1}({\rm Fix}(\eta))) \\
& = & (Z\cup\eta^{-1}({\rm Fix}(\eta)))\setminus{\rm Fix}(\eta) \\
& \supseteq & Z\setminus {\rm Fix}(\eta)
\end{eqnarray*}
Since $\eset(\sigma_\eta)=+\infty$, $F^\Gamma$ and $\Gamma$ are
infinite also for each
$t\geq2$, there exists $\beta\in\Gamma$ such that
$| \{\beta,\varphi(\beta),\ldots,\varphi^t(\beta)\} | = t+1$. So
$\beta,\varphi(\beta),\ldots,\varphi^{t-1}(\beta)\notin{\rm
Fix}(\eta)$, therefore by \ref{boxplusboxplus}, we have
$\mathfrak{v}_\beta,\mathfrak{v}_{\varphi(\beta)},\ldots,\mathfrak{v}_{\varphi^{t-2}(\beta)}=0$,
so $\beta,\varphi(\beta),\ldots,\varphi^{t-2}(\beta)\in
Z\setminus{\rm Fix}(\eta)$, therefore ${\rm
card}(\Gamma\setminus\eta(\rm{supp}(\mathfrak{v})))\geq t-1$, for each
$t\geq2$ and $\Gamma\setminus\eta(\rm{supp}(\mathfrak{v}))$ is
infinite, hence
$\sigma_{\varphi,\mathfrak{w}}^N=\sigma_{\eta,\mathfrak{v}}:F^\Gamma\to
F^\Gamma$ is not finite fibre which leads to the desired  result.
\end{proof}
\begin{corollary}\label{badr4000}
Suppose $\sigma_{\varphi}:F^\Gamma\to F^\Gamma$ is of finite fibre,
$\ecset(\sigma_{\varphi})=+\infty$ and $\eset(\sigma_{\varphi,\mathfrak{w}})=0$,
then $\sigma_{\varphi,\mathfrak{w}}:F^\Gamma\to F^\Gamma$ is not finite fibre.
\end{corollary}
\begin{proof}
By Lemma~\ref{salam10}, Theorem~\ref{ssalam10} and Corollary \ref{badr27}, if $\ecset(\sigma_{\varphi})=+\infty$, then
$\eset(\sigma_{\varphi})=+\infty$. Now Use corollary~\ref{badr6000}  to complete the proof.
\end{proof}
 According to Corollaries~\ref{badr6000} and
\ref{badr4000}, it is not possible to occur the following
properties simultaneously (note that if
$\sigma_{\varphi,\mathfrak{w}}:F^\Gamma\to F^\Gamma$ is of finite
fibre, then $\sigma_{\varphi}:F^\Gamma\to F^\Gamma$ is of finite
fibre too):
\begin{itemize}
\item $\sigma_{\varphi,\mathfrak{w}}:F^\Gamma\to F^\Gamma$ is of finite fibre,
\item $\ecset(\sigma_{\varphi})=+\infty$ or $\eset(\sigma_{\varphi})=+\infty$,
\item $\eset(\sigma_{\varphi,\mathfrak{w}})=0$.
\end{itemize}
However, the following example shows that the next conditions are possible to occur simultaneously:
\begin{itemize}
\item $\sigma_{\varphi}:F^\Gamma\to F^\Gamma$ is of finite fibre,
\item $\ecset(\sigma_{\varphi})=\eset(\sigma_{\varphi})=+\infty$,
\item $\eset(\sigma_{\varphi,\mathfrak{w}})=0$.
\end{itemize}
\begin{exam}\label{badr5000}
Suppose $\Gamma$ is infinite. Choose a one to one sequence
$\{\theta_n\}_{n\geq1}$ in $\Gamma$. Consider
$\varphi:\Gamma\to\Gamma$ with $\varphi(\theta_n)=\theta_{n+1}$
for $n\in\mathbb N$ and $\varphi(\alpha)=\alpha$ for
$\alpha\in\Gamma\setminus\{\theta_n:n\geq1\}$.  Also:
\[\mathfrak{w}_\alpha:=\left\{\begin{array}{lc} 0, & {\rm if \:} \alpha=\theta_{2n}, {\rm \: for \: some \:} n\in\mathbb{N}\:, \\ 1, & {\rm otherwise}\:,
\end{array}\right.\SP \mathfrak{w}=(\mathfrak{w}_\alpha)_{\alpha\in\Gamma}\:.\]
Then $\sigma_{\varphi}:F^\Gamma\to F^\Gamma$ is of finite fibre, $\sigma_{\varphi,\mathfrak{w}}:F^\Gamma\to F^\Gamma$ is not finite fibre,
$\eset(\sigma_{\varphi,\mathfrak{w}})=0$
and $\ecset(\sigma_\varphi)=\eset(\sigma_\varphi)=+\infty$.
\end{exam}
\begin{note}
If $\sigma_{\varphi,\mathfrak{w}}:F^\Gamma\to
F^\Gamma$ is of finite fibre, as it is clear by the following table, then
$\eset(\sigma_{\varphi,\mathfrak{w}})=\eset(\sigma_{\varphi})$.
However, this may fail to be valid for non--finite fibre
$\sigma_{\varphi,\mathfrak{w}}:F^\Gamma\to F^\Gamma$ by
Example~\ref{badr5000}.
\end{note}
\subsection*{A table}
 Now, let's use the following predictions for finite field $F$, infinite set $\Psi$, weight
vector $\mathfrak{v}\in F^\Psi$ and self--map
$\theta:\Psi\to\Psi$:
\begin{center}
\begin{tabular}{cc}
\begin{tabular}{l|}
$\pi_1(\sigma_{\theta,\mathfrak{v}},F^\Psi)$  is $\eset(\sigma_\theta)=0$, \\
$\pi_2(\sigma_{\theta,\mathfrak{v}},F^\Psi)$  is $\eset(\sigma_\theta)=+\infty$, \\
$\pi_3(\sigma_{\theta,\mathfrak{v}},F^\Psi)$  is $\eset(\sigma_{\theta,\mathfrak{v}})=0$, \\
$\pi_4(\sigma_{\theta,\mathfrak{v}},F^\Psi)$  is $\eset(\sigma_{\theta,\mathfrak{v}})=+\infty$,
\end{tabular} &
\begin{tabular}{l}
$\rho_1(\sigma_{\theta,\mathfrak{v}},F^\Psi)$  is $\ecset(\sigma_\theta)=0$, \\
$\rho_2(\sigma_{\theta,\mathfrak{v}},F^\Psi)$  is $\ecset(\sigma_\theta)=+\infty$, \\
$\rho_3(\sigma_{\theta,\mathfrak{v}},F^\Psi)$  is $\ecset(\sigma_{\theta,\mathfrak{v}})=0$, \\
$\rho_4(\sigma_{\theta,\mathfrak{v}},F^\Psi)$  is $\ecset(\sigma_{\theta,\mathfrak{v}})=+\infty$.
\end{tabular}
\\ & 
\end{tabular}
\end{center}
Then we have the following table:
\begin{center}
\includegraphics{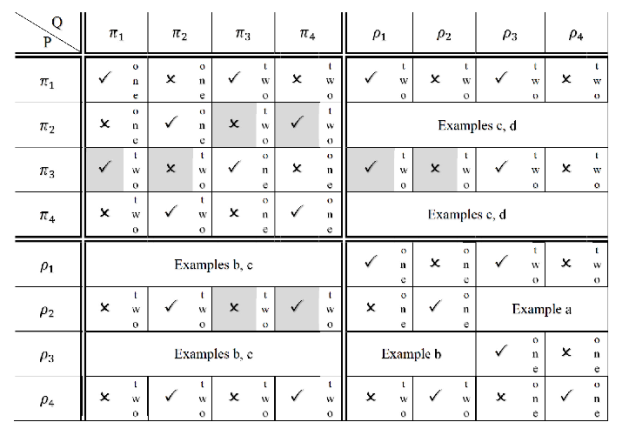}
\\
(Table A)
\end{center}
{\small
\underline{The mark ``$\surd$''} indicates that in the corresponding case for all finite fibre $\sigma_{\theta,\mathfrak{v}}:F^\Psi\to F^\Psi$ we have
``$P(\sigma_{\theta,\mathfrak{v}},F^\Psi)\Rightarrow Q(\sigma_{\theta,\mathfrak{v}},F^\Psi)$''.
\\ \\
\underline{The mark ``$\times$''}  indicates that  in the corresponding case for all finite fibre $\sigma_{\theta,\mathfrak{v}}:F^\Psi\to F^\Psi$ we have
``$P(\sigma_{\theta,\mathfrak{v}},F^\Psi)\Rightarrow\neg Q(\sigma_{\theta,\mathfrak{v}},F^\Psi)$''.
\\ \\
\underline{Vertical ``$\mathsf{xyz}(\in\{\mathsf{one}\:,\:\mathsf{two}\})$''} means that one may find proof of corresponding
``$\surd$'' (hence ``$\times$'') in Step ``$\mathsf{xyz}$'' below.
\\ \\
\underline{Example(s) p, q, ..., r}, means that in the corresponding case in item(s) p, q, ..., r of Counterexample~\ref{mas10}
one may find finite fibre weighted generalized shifts
$\sigma_{\lambda,\mathfrak{y}},\sigma_{\lambda',\mathfrak{y}'}:F^\Psi\to
F^\Psi$ such that ``$P(\sigma_{\lambda,\mathfrak{y}},F^\Psi)\wedge
Q(\sigma_{\lambda,\mathfrak{y}},F^\Psi)$'' and
``$P(\sigma_{\lambda',\mathfrak{y}'},F^\Psi)\wedge \neg
Q(\sigma_{\lambda',\mathfrak{y}'},F^\Psi)$''.
\\ \\
\underline{In gray boxes} if we substitute the assumption of finite fibreness of  $\sigma_{\theta,\mathfrak{v}}:F^\Psi\to F^\Psi$ 
just by finite fibreness of $\sigma_{\theta}:F^\Psi\to F^\Psi$, then the result in the corresponding case in the above table
may fail to be true according to  Example~\ref{badr5000}.
}
\subsection*{Proof of Statements in Table A}
Let's present proof through the following steps and examples:
\\
{\bf  Step one.} It's clear that for each proposition ``$\alpha$'' we have ``$\alpha\Rightarrow\alpha$''. Moreover
\begin{eqnarray*}
\pi_1(\sigma_{\theta,\mathfrak{v}},F^\Psi)\equiv\neg\pi_2(\sigma_{\theta,\mathfrak{v}},F^\Psi)
&\:,\:&\pi_3(\sigma_{\theta,\mathfrak{v}},F^\Psi)\equiv\neg\pi_4(\sigma_{\theta,\mathfrak{v}},F^\Psi)\:,
\\
\rho_1(\sigma_{\theta,\mathfrak{v}},F^\Psi)\equiv\neg\rho_2(\sigma_{\theta,\mathfrak{v}},F^\Psi)&\:,\:&
\rho_3(\sigma_{\theta,\mathfrak{v}},F^\Psi)\equiv\neg\rho_4(\sigma_{\theta,\mathfrak{v}},F^\Psi)\:.
\end{eqnarray*}
{\bf  Step two.}
We have the following implications  for finite fibre
$\sigma_{\theta,\mathfrak{v}}:F^\Psi\to F^\Psi$:
\begin{description}
\item[i:]
	``$\pi_1(\sigma_{\theta,\mathfrak{v}},F^\Psi)\Rightarrow\pi_3(\sigma_{\theta,\mathfrak{v}},F^\Psi)$''.
	Suppose $\eset(\sigma_\theta)=0$, then by Theorem~\ref{ssalam10}
	(or~\cite[Theorem 2.4]{nili}), $\sigma_\theta$ and in its consequence  $\theta$ is quasi--periodic. By
	Lemma~\ref{salam10}, $\sigma_{\theta,\mathfrak{v}}$ is
	quasi--periodic too and by Theorem~\ref{ssalam10},
	$\eset(\sigma_{\theta,\mathfrak{v}})=0$.
\item[ii:] ``$\pi_3(\sigma_{\theta,\mathfrak{v}},F^\Psi)\Rightarrow\rho_3(\sigma_{\theta,\mathfrak{v}},F^\Psi)$''.
	Use Lemma~\ref{salam10}, Theorem~\ref{ssalam10} and Corollary \ref{badr27}.
\item[iii:] ``$\pi_1(\sigma_{\theta,\mathfrak{v}},F^\Psi)\Rightarrow\rho_1(\sigma_{\theta,\mathfrak{v}},F^\Psi)$''.
	If $\sigma_{\theta,\mathfrak{v}}$ is of finite fibre, then $\sigma_\theta= \sigma_{\theta,\mathfrak{u}}$ is of finite fibre too,
	hence by item (ii) we have ``$\pi_3(\sigma_{\theta,\mathfrak{u}},F^\Psi)\Rightarrow\rho_3(\sigma_{\theta,		
	\mathfrak{u}},F^\Psi)$'', i.e. ``$\pi_1(\sigma_{\theta,\mathfrak{v}},F^\Psi)\Rightarrow\rho_1(\sigma_{\theta,		
	\mathfrak{v}},F^\Psi)$''.
\item[iv:] ``$\pi_1(\sigma_{\theta,\mathfrak{v}},F^\Psi)\Rightarrow\rho_3(\sigma_{\theta,\mathfrak{v}},F^\Psi)$''. 
	Use items (i) and (ii).
\item[v:] ``$\pi_4(\sigma_{\theta,\mathfrak{v}},F^\Psi)\Rightarrow\pi_2(\sigma_{\theta,\mathfrak{v}},F^\Psi)$''.
	Use (i) and Step one.
\item[vi:]  ``$\rho_4(\sigma_{\theta,\mathfrak{v}},F^\Psi)\Rightarrow\pi_4(\sigma_{\theta,\mathfrak{v}},F^\Psi)$''.
	Use (ii) and Step one.
\item[vii:]  ``$\rho_2(\sigma_{\theta,\mathfrak{v}},F^\Psi)\Rightarrow\pi_2(\sigma_{\theta,\mathfrak{v}},F^\Psi)$''.
	Use (iii) and Step one.
\item[viii:]  ``$\rho_4(\sigma_{\theta,\mathfrak{v}},F^\Psi)\Rightarrow\pi_2(\sigma_{\theta,\mathfrak{v}},F^\Psi)$''.
	Use (iv) and Step one.
\item[ix:] ``$\rho_1(\sigma_{\theta,\mathfrak{v}},F^\Psi)\Rightarrow\rho_3(\sigma_{\theta,\mathfrak{v}},F^\Psi)$''.
	Suppose $\ecset(\sigma_\theta)=0$, then by Theorem~\ref{badr3000}
	(or~\cite[Corollary 3.9]{nili}),  all points of $\Psi$ are
	quasi--periodic w.r.t.  $\theta$ and $\sup\{{\rm
	per}(\alpha):\alpha\in{\rm Per}(\theta)\} \in\mathbb{N}$. In
	particular, all points of $L:=\Gamma\setminus \{\alpha\in\Psi:\exists
	n\geq0\:\mathfrak{v}_{\theta^n(\alpha)}=0\}$ are quasi--periodic
	w.r.t.  $\theta$ and $\sup(\{{\rm per}(\alpha):\alpha\in{\rm
	Per}(\theta)\cap L\}\cup\{1\}) \in\mathbb{N}$. Again by
	Theorem~\ref{badr3000}, $\ecset(\sigma_{\theta,\mathfrak{v}})=0$.
\item[x:] ``$\rho_4(\sigma_{\theta,\mathfrak{v}},F^\Psi)\Rightarrow\rho_2(\sigma_{\theta,\mathfrak{v}},F^\Psi)$''.
	Use (ix) and Step one.
\item[xi:]  ``$\pi_3(\sigma_{\theta,\mathfrak{v}},F^\Psi)\Rightarrow\pi_1(\sigma_{\theta,\mathfrak{v}},F^\Psi)$''.
	Use Corollary~\ref{badr6000}.
\item[xii:] ``$\pi_2(\sigma_{\theta,\mathfrak{v}},F^\Psi)\Rightarrow\pi_4(\sigma_{\theta,\mathfrak{v}},F^\Psi)$''.
	Use (xi) and Step one.
\item[xiii:] ``$\rho_2(\sigma_{\theta,\mathfrak{v}},F^\Psi)\Rightarrow\pi_4(\sigma_{\theta,\mathfrak{v}},F^\Psi)$''.
	Use Corollary~\ref{badr4000}.
\item[xiv:]  ``$\pi_3(\sigma_{\theta,\mathfrak{v}},F^\Psi)\Rightarrow\rho_1(\sigma_{\theta,\mathfrak{v}},F^\Psi)$''.
	Use (xiii) and Step one.
\end{description}
\begin{exam}\label{mas10}
Suppose $\mathfrak{u}:=(1)_{\alpha\in\Psi}$ and consider a one to
one double sequence $\{\beta_{(n,m)}\}_{n,m\geq1}$ in $\Psi$.
Consider:
\begin{itemize}
\item $\theta_1:\Psi\to\Psi$ with:
    \[\theta_1(\alpha):=\left\{\begin{array}{lc}
    \beta_{(1,n+1)}, & {\rm if \:}\alpha=\beta_{(1,n)}, {\rm \: for \: some \:} 2^k\leq n\leq 2^{k+1}-1 {\rm \: and \:} k\geq1 \: , \\
    \beta_{(1,2^k)}, & {\rm if \:} \alpha=\beta_{(1,2^{k+1}-1)}, {\rm \: for \: some \:} k\geq1 \: , \\
    \beta_{(k,n)}, & {\rm if \:} \alpha=\beta_{(k+1,n)},{\rm \: for \: some \:} n,k\geq1 \: , \\
    \alpha, &otherwise\:, \end{array} \right.\]
    and:
    \[\mathfrak{v}_\alpha=\left\{\begin{array}{lc} 0, & {\rm if \:} \alpha=\beta_{(1,n)}, {\rm \: for \: some \:} n\geq1\:, \\ 1, & {\rm otherwise}\:,
    \end{array}\right.\SP ,\SP \mathfrak{v}=(\mathfrak{v}_\alpha)_{\alpha\in\Psi}\:.\]
    Then $\theta_1(\Psi\setminus\{\alpha\in\Psi:\mathfrak{v}_\alpha=0\})=\theta_1(\Psi)=\Psi$ and
    $\sigma_{\theta_1,\mathfrak{v}}:F^\Psi\to F^\Psi$ is of finite fibre.
\item $\theta_2:\Psi\to\Psi$ with $\theta_2(\beta_{(1,m+1)})=\beta_{(1,m)}$ for $m\geq1$,
    $\theta_2(\alpha)=\alpha$, otherwise.
\item $\theta_3:\Psi\to\Psi$ with $\theta_3(\beta_{(1,2m-1)})=\beta_{(1,2m+1)}$ and
    $\theta_3(\beta_{(1,2m+2)})=\beta_{(1,2m)}$ for $m\geq1$, also $\theta_3(\beta_{(1,2)})=\beta_{(1,1)}$ and
    $\theta_3(\alpha)=\alpha$, otherwise.
\end{itemize}
Then
\begin{itemize}
\item[a.] $\ecset(\sigma_{\theta_1})=+\infty$, however $\ecset(\sigma_{\theta_1,\mathfrak{u}})=+\infty$ and $\ecset(\sigma_{\theta_1,\mathfrak{v}})=0$
\item[b.] $\ecset(\sigma_{id_\Psi,\mathfrak{u}})=\ecset(\sigma_{id_\Psi})=\eset(\sigma_{id_\Psi,\mathfrak{u}})=\eset(\sigma_{id_\Psi})=0$, also
$\ecset(\sigma_{\theta_1,\mathfrak{v}})=0$ and $\ecset(\sigma_{\theta_1})=+\infty$
\item[c.]  $\eset(\sigma_{\theta_2})=\eset(\sigma_{\theta_2,\mathfrak{u}})=+\infty$ and $\ecset(\sigma_{\theta_2})=\ecset(\sigma_{\theta_2,\mathfrak{u}})=0$
\item[d.]  $\eset(\sigma_{\theta_3})=\ecset(\sigma_{\theta_3})=\ecset(\sigma_{\theta_3,\mathfrak{u}})=+\infty$
\end{itemize}
\end{exam}
\section{Set--theoretical entropy of $\sigma_{\varphi,\mathfrak{w}}\restriction_{\mathop{\bigoplus}\limits_{\Gamma}F}$ whenever
\\
$\mathop{\bigoplus}\limits_{\Gamma}F$ is $\sigma_{\varphi,\mathfrak{w}}-$invariant}\label{maryam40}
\noindent In this section we pay attention to the restriction of $\sigma_{\varphi,\mathfrak{w}}$ to direct sum 
$\mathop{\bigoplus}\limits_{\Gamma}F:=\{x\in F^\Gamma:{\rm supp}(x)$ is finite$\}$. 
We try to find out all conditions under which $\sigma_{\varphi,\mathfrak{w}}(\mathop{\bigoplus}\limits_{\Gamma}F)\subseteq \mathop{\bigoplus}\limits_{\Gamma}F$ and in the above case we show 
$\eset(\sigma_{\varphi,\mathfrak{w}}\restriction_{\mathop{\bigoplus}\limits_{\Gamma}F})\in\{0,+\infty\}$ where
$\eset(\sigma_{\varphi,\mathfrak{w}}\restriction_{\mathop{\bigoplus}\limits_{\Gamma}F})=+\infty$ if and only if there exists
a $\varphi-$anti--orbit one to one sequence in ${\rm supp}(\mathfrak{w})$.
\\
In this section for each $\beta\in\Gamma$ let
$\delta_{\beta,\beta}=1$ and $\delta_{\alpha,\beta}=0$ for $\alpha\neq\beta$
also
\[{\mathsf e}_\beta:=(\delta_{\alpha,\beta})_{\alpha\in\Gamma}\:.\]
\begin{lemma}\label{kheybar30}
The following statements are equivalent:
\begin{itemize}
\item[1.] $\sigma_{\varphi,\mathfrak{w}}(\mathop{\bigoplus}\limits_{\Gamma}F)\subseteq \mathop{\bigoplus}\limits_{\Gamma}F$,
\item[2.] for each $\beta\in\Gamma$, $\sigma_{\varphi,\mathfrak{w}}({\mathsf e}_\beta)\in \mathop{\bigoplus}\limits_{\Gamma}F$,
\item[3.] $\varphi\restriction_{{\rm supp}(\mathfrak{w})}:{\rm supp}(\mathfrak{w})\to\Gamma$ is of finite fibre.
\end{itemize}
\end{lemma}
\begin{proof}
(1 $\Leftrightarrow$ 2): It is obvious, since $\sigma_{\varphi,\mathfrak{w}}:F^\Gamma\to F^\Gamma$ is a linear map and
$\mathop{\bigoplus}\limits_{\Gamma}F$ is the linear subspace (of $F^\Gamma$) generated by $\{{\mathsf e}_\alpha:\alpha\in\Gamma\}$.
\\
(2 $\Leftrightarrow$ 3): For each $\beta\in \Gamma$ we have
\begin{eqnarray*}
{\rm supp}(\sigma_{\varphi,\mathfrak{w}}({\mathsf e}_\beta)) & = & 
	{\rm supp}((\mathfrak{w}_\alpha\delta_{\varphi(\alpha),\beta})_{\alpha\in\Gamma})
	=\{\alpha\in\Gamma:\mathfrak{w}_\alpha\delta_{\varphi(\alpha),\beta}\neq0\} \\
& = & \{\alpha\in{\rm supp}(\mathfrak{w}):\delta_{\varphi(\alpha),\beta}\neq0\}
	=\{\alpha\in{\rm supp}(\mathfrak{w}):\varphi(\alpha)=\beta\} \\
& = & {\rm supp}(\mathfrak{w})\cap\varphi^{-1}(\beta)=\varphi\restriction_{{\rm supp}(\mathfrak{w})}^{-1}(\beta) 
\end{eqnarray*}
hence $\sigma_{\varphi,\mathfrak{w}}({\mathsf e}_\beta)\in \mathop{\bigoplus}\limits_{\Gamma}F$ if and only if
$\varphi\restriction_{{\rm supp}(\mathfrak{w})}^{-1}(\beta) $ is finite, which leads to the desired result.
\end{proof}
\begin{convention}\label{kheybar75}
Henceforth, suppose $\sigma_{\varphi,\mathfrak{w}}(\mathop{\bigoplus}\limits_{\Gamma}F)\subseteq \mathop{\bigoplus}\limits_{\Gamma}F$
(or equivalently, by Lemma~\ref{kheybar30},  $\varphi\restriction_{{\rm supp}(\mathfrak{w})}:{\rm supp}(\mathfrak{w})\to\Gamma$ is of finite fibre).
\end{convention}
\begin{lemma}\label{kheybar60}
If $\beta\in\Gamma$ and $\{\sigma_{\varphi,\mathfrak{w}}^n(\mathsf{e}_\beta)\}_{n\geq1}$ is a one to one sequence, then 
\begin{itemize}
\item[1.] ${\rm supp}(\mathfrak{w})\cap\varphi^{-1}(\beta)\setminus\{\beta\}\neq\varnothing$ and
	$\sigma_{\varphi,\mathfrak{w}}(\mathsf{e}_\beta)=\mathop{\Sigma}\limits_{\kappa\in {\rm supp}(\mathfrak{w})\cap\varphi^{-1}(\beta)}\mathfrak{w}_\kappa\mathsf{e}_\kappa$,
\item[2.] there 
exists $\mu\in{\rm supp}(\mathfrak{w})\cap\varphi^{-1}(\beta)\setminus\{\beta\}$ 
such that $\{\sigma_{\varphi,\mathfrak{w}}^n(\mathsf{e}_\mu)\}_{n\geq1}$ is a one to one sequence,
\item[3.] there exists an infinite $\varphi-$anti--orbit $\{\alpha_n\}_{n\geq1}$ in ${\rm supp}(\mathfrak{w})$.
\end{itemize}
\end{lemma}
\begin{proof}
Suppose $\{\sigma_{\varphi,\mathfrak{w}}^n(\mathsf{e}_\beta)\}_{n\geq1}$ is a one to one sequence. 
\\
{\bf 1)} If $\sigma_{\varphi,\mathfrak{w}}(\mathsf{e}_\beta)=(0)_{\alpha\in\Gamma}$, then 
$\sigma_{\varphi,\mathfrak{w}}^n(\mathsf{e}_\beta)=(0)_{\alpha\in\Gamma}$ for all $n\geq1$, which is a contradiction.
Thus $\sigma_{\varphi,\mathfrak{w}}(\mathsf{e}_\beta)\neq(0)_{\alpha\in\Gamma}$ and
\[(0)_{\alpha\in\Gamma}\neq\sigma_{\varphi,\mathfrak{w}}(\mathsf{e}_\beta)=
	\sigma_{\varphi,\mathfrak{w}}((\delta_{\alpha,\beta})_{\alpha\in\Gamma})=
	(\mathfrak{w}_\alpha \delta_{{\varphi(\alpha)},\beta})_{\alpha\in\Gamma}=
	\mathop{\Sigma}\limits_{\kappa\in {\rm supp}(\mathfrak{w})\cap\varphi^{-1}(\beta)}\mathfrak{w}_\kappa\mathsf{e}_\kappa\]
therefore ${\rm supp}(\mathfrak{w})\cap\varphi^{-1}(\beta)\neq\varnothing$. 
\\
If ${\rm supp}(\mathfrak{w})\cap\varphi^{-1}(\beta)=\{\beta\}$,
then $\sigma_{\varphi,\mathfrak{w}}(\mathsf{e}_\beta)=\mathop{\Sigma}\limits_{\kappa\in {\rm supp}(\mathfrak{w})\cap\varphi^{-1}(\beta)}\mathfrak{w}_\kappa\mathsf{e}_\kappa=\mathfrak{w}_\beta \mathsf{e}_\beta$, which leads to (by induction on $n\geq1$)
\[\forall n\geq1\SP \sigma_{\varphi,\mathfrak{w}}^n(\mathsf{e}_\beta)=\mathfrak{w}_\beta^n \mathsf{e}_\beta\:,\]
thus $\{\sigma_{\varphi,\mathfrak{w}}^n(\mathsf{e}_\beta):n\geq1\}$ is an infinite subset of finite set
$\{r\mathsf{e}_\beta:r\in F\}$ which is a contradiction, therefore ${\rm supp}(\mathfrak{w})\cap\varphi^{-1}(\beta)\neq\{\beta\}$.
\\
{\bf 2)} Suppose ${\rm supp}(\mathfrak{w})\cap\varphi^{-1}(\beta)
=\{\kappa_1,\ldots,\kappa_p\}$ has $p$ elements. 
Then
\[\{\sigma_{\varphi,\mathfrak{w}}^n(\mathsf{e}_\beta):n\geq1\}
=\{\mathfrak{w}_{\kappa_1}\sigma_{\varphi,\mathfrak{w}}^n(\mathsf{e}_{\kappa_1})+\cdots+
\mathfrak{w}_{\kappa_p}\sigma_{\varphi,\mathfrak{w}}^n(\mathsf{e}_{\kappa_p}):n\geq0\}\]
is an infinite subset of 
$\{\mathfrak{w}_{\kappa_1}\sigma_{\varphi,\mathfrak{w}}^{n_1}(\mathsf{e}_{\kappa_1})+\cdots+
\mathfrak{w}_{\kappa_p}\sigma_{\varphi,\mathfrak{w}}^{n_p}(\mathsf{e}_{\kappa_p}):n_1,\ldots,n_p\geq0\}$.
So $\{\mathfrak{w}_{\kappa_1}\sigma_{\varphi,\mathfrak{w}}^{n_1}(\mathsf{e}_{\kappa_1})+\cdots+
\mathfrak{w}_{\kappa_p}\sigma_{\varphi,\mathfrak{w}}^{n_p}(\mathsf{e}_{\kappa_p}):n_1,\ldots,n_p\geq0\}$ is infinite too and
there exists $\theta\in\{\kappa_1,\ldots,\kappa_p\}$ such that
$\{\sigma_{\varphi,\mathfrak{w}}^{n}(\mathsf{e}_{\theta}):n\geq0\}$ is an infinite set. 
\\
Suppose for each
$\mu\in {\rm supp}(\mathfrak{w})\cap\varphi^{-1}(\beta)$ either $\mu=\beta$ or 
$\{\sigma_{\varphi,\mathfrak{w}}^{n}(\mathsf{e}_{\mu}):n\geq0\}$  is finite, then by the above discussion $\theta=\beta$
and (by item (1)) $p\geq2$. So we may suppose $\theta=\beta=\kappa_1$, moreover,
$D=\{\mathsf{e}_\beta\}\cup\{\sigma_{\varphi,\mathfrak{w}}^{n}(\mathsf{e}_{\kappa_i}):2\leq i\leq p, n\geq0\}$
is a finite subset of $\mathop{\bigoplus}\limits_\Gamma F$ and $V$ is the linear subspace generated by $D$, i.e.
\linebreak
$V:=\{r_1x_1+\cdots+r_m x_m:m\geq1, x_1,\ldots,x_m\in D,r_1,\ldots,r_m\in F\}$.
$V$ is a linear space with finite generator $D$ over finite field $F$, thus $V$ is finite.
Using induction on $n\geq1$ we prove $\sigma_{\varphi,\mathfrak{w}}^n(\mathsf{e}_{\beta})$ belongs to $V$.
Note that
\[\sigma_{\varphi,\mathfrak{w}}(\mathsf{e}_{\beta})\mathop{=}\limits^{(1)}\mathfrak{w}_{\kappa_1}\mathsf{e}_{\kappa_1}+\cdots+
\mathfrak{w}_{\kappa_p}\mathsf{e}_{\kappa_p}\mathop{=}\limits^{\kappa_1=\beta}\mathfrak{w}_{\beta}\mathsf{e}_{\beta}+
\cdots+\mathfrak{w}_{\kappa_p}\mathsf{e}_{\kappa_p}\in V\:.\]
For $t\geq1$ if 
$\sigma_{\varphi,\mathfrak{w}}^t(\mathsf{e}_{\beta})\in V$, then 
there exists $r_0,r_1,\ldots,r_m\in F$, 
\linebreak
$\lambda_1,\ldots,\lambda_m\in\{\kappa_2,\ldots,\kappa_p\}$ and
$n_1,\ldots, n_m\geq0$ such that
\[\sigma_{\varphi,\mathfrak{w}}^t(\mathsf{e}_{\beta})=r_0\mathsf{e}_{\beta}+
r_1\sigma_{\varphi,\mathfrak{w}}^{n_1}(\mathsf{e}_{\lambda_1})+\cdots+
r_m\sigma_{\varphi,\mathfrak{w}}^{n_m}(\mathsf{e}_{\lambda_m})\:,\]
then
using $\sigma_{\varphi,\mathfrak{w}}(\mathsf{e}_{\beta}),
	\sigma_{\varphi,\mathfrak{w}}^{n_1+1}(\mathsf{e}_{\lambda_1}),\ldots,
	\sigma_{\varphi,\mathfrak{w}}^{n_m+1}(\mathsf{e}_{\lambda_m}) \in V$ we have:
\begin{eqnarray*}
\sigma_{\varphi,\mathfrak{w}}^{t+1}(\mathsf{e}_{\beta}) & = & \sigma_{\varphi,\mathfrak{w}}(r_0\mathsf{e}_{\beta}+
	r_1\sigma_{\varphi,\mathfrak{w}}^{n_1}(\mathsf{e}_{\lambda_1})+\cdots+
	r_m\sigma_{\varphi,\mathfrak{w}}^{n_m}(\mathsf{e}_{\lambda_m}))  \\
& = & r_0\sigma_{\varphi,\mathfrak{w}}(\mathsf{e}_{\beta})+
	r_1\sigma_{\varphi,\mathfrak{w}}^{n_1+1}(\mathsf{e}_{\lambda_1})+\cdots+
	r_m\sigma_{\varphi,\mathfrak{w}}^{n_m+1}(\mathsf{e}_{\lambda_m}) \in V
\end{eqnarray*}
which completes the steps of induction. 
$\{\sigma_{\varphi,\mathfrak{w}}^n(\mathsf{e}_{\beta}):n\geq1\}$ is an infinite subset of finite set $V$ which is a contradiction.
\\
Therefore there exists $\mu\in {\rm supp}(\mathfrak{w})\cap\varphi^{-1}(\beta)\setminus\{\beta\}$ such that
$\{\sigma_{\varphi,\mathfrak{w}}^{n}(\mathsf{e}_{\mu}):n\geq0\}$  is infinite, hence 
$\{\sigma_{\varphi,\mathfrak{w}}^{n}(\mathsf{e}_{\mu})\}_{n\geq1}$ is a one to one sequence.
\\
{\bf 3)} We have the following cases:
\\
{\it Case 1.} $\beta\notin{\rm Per}(\varphi)$. Choose $\{\alpha_n\}_{n\geq1}$ inductively in the following way:
\\
$\bullet$ by item (1) choose $\alpha_1\in{\rm supp}(\mathfrak{w})\cap\varphi^{-1}(\beta)$ such that $\{\sigma_{\varphi,\mathfrak{w}}^n(\mathsf{e}_{\alpha_1})\}_{n\geq1}$ is a one to one sequence,
\\
$\bullet$ for $k\geq1$ suppose $\alpha_1,\ldots,\alpha_k\in{\rm supp}(\mathfrak{w})$ have been chosen such that
$\varphi(\alpha_i)=\alpha_{i-1}$ for $1<i\leq k$ and 
$\{\sigma_{\varphi,\mathfrak{w}}^{n}(\mathsf{e}_{\alpha_k})\}_{n\geq1}$ is a one to one sequence.
By item (1) choose $\alpha_{k+1}\in{\rm supp}(\mathfrak{w})\cap\varphi^{-1}(\alpha_k)$ such that $\{\sigma_{\varphi,\mathfrak{w}}^n(\mathsf{e}_{\alpha_{k+1}})\}_{n\geq1}$ is a one to one sequence.
\\
 Using the above inductive construction, $\{\alpha_n\}_{n\geq1}$ is a $\varphi-$anti--orbit sequence in
 ${\rm supp}(\mathfrak{w})$ with  $\varphi(\alpha_1)=\beta$. We claim that $\{\alpha_n\}_{n\geq1}$ is a one to one sequence too. 
 Consider $n>m\geq1$ such that $\alpha_n=\alpha_m$, then $\beta=\varphi^n(\alpha_n)=\varphi^n(\alpha_m)=\varphi^{n-m}(\varphi^m(\alpha_m))
 =\varphi^{n-m}(\beta)$, which is in contradiction with $\beta\notin{\rm Per}(\varphi)$. Therefore
 $\{\alpha_n\}_{n\geq1}$ is a one to one $\varphi-$anti--orbit sequence in
 ${\rm supp}(\mathfrak{w})$.
 \\
{\it Case 2.} $\beta\in {\rm Per}(\varphi)$ with ${\rm per}(\beta)=t\geq1$.
$\{\sigma_{\varphi,\mathfrak{w}}^n(\mathsf{e}_\beta)\}_{n\geq1}$ is a one to one sequence so
$\{\sigma_{\varphi,\mathfrak{w}}^{nt}(\mathsf{e}_\beta)\}_{n\geq1}=\{(\sigma_{\varphi,\mathfrak{w}}^t)^n(\mathsf{e}_\beta)\}_{n\geq1}$ 
is a one to one sequence too.
Let $\eta=\varphi^t$ and $\mathfrak{v}=(\mathfrak{v}_\alpha)_{\alpha\in\Gamma}:=\mathfrak{w}\sigma_\varphi(\mathfrak{w})\cdots
\sigma_{\varphi^{t-1}}(\mathfrak{w})$, by Lemma~\ref{kheybar50} we have $\sigma_{\varphi,\mathfrak{w}}^t=\sigma_{\eta,\mathfrak{v}}$.
Thus $\{\sigma_{\eta,\mathfrak{v}}^n(\mathsf{e}_\beta)\}_{n\geq1}$ is a one to one sequence. By item (2) there exists
$\mu\in{\rm supp}(\mathfrak{v})\cap\eta^{-1}(\beta)\setminus\{\beta\}$ such that 
$\{\sigma_{\eta,\mathfrak{v}}^n(\mathsf{e}_\mu)\}_{n\geq1}$ is a one to one sequence.
For each $n\geq1$ we have $\eta^n(\mu)=\eta^{n-1}(\eta(\mu))=\eta^{n-1}(\beta)=\varphi^{t(n-1)}(\beta)=\beta\neq\mu$ hence
$\mu\notin{\rm Per}(\eta)$. By Case 1 there exists a one to one $\eta-$anti--orbit sequence $\{\mu_n\}_{n\geq1}$ in
${\rm supp}(\mathfrak{v})$. Let:
\[\begin{array}{llcl}
\alpha_1:=\mu_1=\varphi^t(\mu_2), & \alpha_2:=\varphi^{t-1}(\mu_2),&\cdots,& \alpha_t:=\varphi(\mu_2) ,\\
\alpha_{t+1}:=\mu_2=\varphi^t(\mu_3), & \alpha_{t+2}:=\varphi^{t-1}(\mu_3),&\cdots,& \alpha_{2t}:=\varphi(\mu_3), \\
\vdots & & & 
\end{array}\]
i.e.,  $\alpha_{it+j}=\varphi^{t-j+1}(\mu_{i+2})$ for each $i\geq0$ and $j\in\{1,\ldots,t\}$.
Clearly $\{\alpha_n\}_{n\geq1}$ is a $\varphi-$anti--orbit sequence with infinite sub--sequence $\{\mu_n\}_{n\geq1}$.
Hence $\{\alpha_n\}_{n\geq1}$ is an infinite  $\varphi-$anti--orbit sequence, therefore it is
a one to one $\varphi-$anti--orbit sequence. 
\\
Also for each $\mu\in{\rm supp}(\mathfrak{v})$  we have $0\neq\mathfrak{v}_\mu=\mathfrak{w}_\mu\mathfrak{w}_{\varphi(\mu)}\cdots
\mathfrak{w}_{\varphi^{t-1}(\mu)}$, therefore 
\linebreak
$\mathfrak{w}_\mu\neq0,\mathfrak{w}_{\varphi(\mu)}\neq0,\ldots,
\mathfrak{w}_{\varphi^{t-1}(\mu)}\neq0$, hence $\mu,\varphi(\mu),\ldots,\varphi^{t-1}(\mu)\in{\rm supp}(\varphi)$,
so $\alpha_n\in{\rm supp}(\mathfrak{w})$ for each $n\geq1$. Hence
$\{\alpha_n\}_{n\geq1}$ is a one to one  $\varphi-$anti--orbit sequence in ${\rm supp}(\mathfrak{w})$.  
\end{proof}
\begin{theorem}\label{kheybar70}
The following statements are equivalent:
\begin{itemize}
\item[1.] there exists an infinite $\varphi-$anti--orbit $\{\alpha_n\}_{n\geq1}$ in ${\rm supp}(\mathfrak{w})$,
\item[2.] $\eset(\sigma_{\varphi,\mathfrak{w}}\restriction_{\mathop{\bigoplus}\limits_{\Gamma}F})=+\infty$,
\item[3.] $\eset(\sigma_{\varphi,\mathfrak{w}}\restriction_{\mathop{\bigoplus}\limits_{\Gamma}F})>0$,
\item[4.] there exists $\beta\in\Gamma$ such that $\{\sigma_{\varphi,\mathfrak{w}}^n(\mathsf{e}_\beta)\}_{n\geq1}$ is a
	one to one sequence.
\end{itemize}
$\:$\\
So:
{\small \[\eset(\sigma_{\varphi,\mathfrak{w}}\restriction_{\mathop{\bigoplus}\limits_{\Gamma}F})=\left\{\begin{array}{lc}
+\infty, & if \: there \: exists \: a \: one \: to \: one \: \varphi\!-\!anti-orbit \: sequence \: in \: {\rm supp}(\mathfrak{w}), \\
0\:, & otherwise\:.
\end{array}\right.\]}
$\;$
\end{theorem}
\begin{proof}
(1 $\Rightarrow$ 2): Suppose $\{\alpha_n\}_{n\geq1}$ is an infinite $\varphi-$anti--orbit  in ${\rm supp}(\mathfrak{w})$.
For $m\geq1$ let:
\[x^m_\alpha:=\left\{\begin{array}{lc} 1, & \alpha=\alpha_1, \alpha_{m+2} \:, \\ 0, & {\rm otherwise}\:,
\end{array}\right.\SP
{\rm and}\SP x^m:=\mathsf{e}_{\alpha_1}+\mathsf{e}_{\alpha_{m+2}}=(x^m_\alpha)_{\alpha\in\Gamma}.\] 
For
convenience let
\[\sigma_{\varphi,\mathfrak{w}}^i(x^j)=(y^{j,i}_\alpha)_{\alpha\in\Gamma}\SP(i,j\geq1)\:.\]
Therefore for each $i,j\geq1$ we have $(y^{j,i+1}_\alpha)_{\alpha\in\Gamma}=\sigma_{\varphi,\mathfrak{w}}((y^{j,i}_\alpha)_{\alpha\in\Gamma})=(\mathfrak{w}_\alpha y^{j,i}_{\varphi(\alpha)})_{\alpha\in\Gamma})$, in particular
\[\forall i,j,k\geq1 \SP y^{j,i+1}_{\alpha_{k+1}}=\mathfrak{w}_{\alpha_{k+1}} y^{j,i}_{\varphi(\alpha_{k+1})}
=\mathfrak{w}_{\alpha_{k+1}} y^{j,i}_{\alpha_k}\:.\]
Hence for $m\geq1$ we have:
{\small\[\begin{array}{lcl}
(x^{m}_{\alpha_k})_{k\geq1} & = & (1, \underbrace{0,\cdots,0}_{m\: times},1,0,0,\cdots) \\ && \\
(y^{m,1}_{\alpha_k})_{k\geq1} & = & (y^{m,1}_{\alpha_1},\mathfrak{w}_{\alpha_2},\underbrace{0,\cdots,0}_{m\: times},\mathfrak{w}_{\alpha_{m+3}},0,0,\cdots) \\
& \vdots & \\
(y^{m,n}_{\alpha_k})_{k\geq1} & = & (y^{m,n}_{\alpha_1},\cdots, y^{m,n}_{\alpha_n}, \underbrace{\mathfrak{w}_{\alpha_2}\mathfrak{w}_{\alpha_3}\cdots\mathfrak{w}_{\alpha_{n+1}}}_{\neq0},\underbrace{0,\cdots,0}_{m\: times},\underbrace{\mathfrak{w}_{\alpha_{m+3}}\cdots\mathfrak{w}_{\alpha_{m+n+2}}}_{\neq0},0,0,\cdots) \\
\end{array}\]}
 We claim that $\{\{\sigma_{\varphi,\mathfrak{w}}^n(x^m)\}_{n\geq1}:m\geq1\}$ is a collection of pairwise disjoint one to one sequences. 
For this
aim, consider $(p,q),(s,t)\in{\mathbb N}\times{\mathbb N}$ with
$\sigma_{\varphi,\mathfrak{w}}^q(x^p)= \sigma_{\varphi,\mathfrak{w}}^t(x^s)$.
By $(y^{p,q}_\alpha)_{\alpha\in\Gamma}=\sigma_{\varphi,\mathfrak{w}}^q(x^p)= \sigma_{\varphi,\mathfrak{w}}^t(x^s)=(y^{s,t}_\alpha)_{\alpha\in\Gamma}$ we have:
\begin{equation}\label{kheybar10}
p+q+2=\max\{k\geq1:y^{p,q}_{\alpha_k}\neq0\}=\max\{k\geq1:y^{s,t}_{\alpha_k}\neq0\}=s+t+2
\end{equation}
thus
\begin{equation}\label{kheybar20}
\begin{array}{rcl}
q+1 & = & \max\{k\geq1:y^{p,q}_{\alpha_k}\neq0,k\neq p+q+2\} \\
& = &  \max\{k\geq1:y^{p,q}_{\alpha_k}\neq0,k\neq s+t+2\}\\
& = & \max\{k\geq1:y^{s,t}_{\alpha_k}\neq0,k\neq s+t+2\}=t+1\:.
\end{array}
\end{equation}
Equations \ref{kheybar10} and \ref{kheybar20} lead us to $(p,q)=(s,t)$.
Thus $\{\{\sigma_{\varphi,\mathfrak{w}}\restriction_{\mathop{\bigoplus}\limits_{\Gamma}F}^n(x^m)\}_{n\geq1}:m\geq1\}$ is a collection of pairwise disjoint one to one sequences
and $+\infty={\mathsf o}(\sigma_{\varphi,\mathfrak{w}}\restriction_{\mathop{\bigoplus}\limits_{\Gamma}F})=\eset(\sigma_{\varphi,\mathfrak{w}}\restriction_{\mathop{\bigoplus}\limits_{\Gamma}F})$.
\\
(2 $\Rightarrow$ 3): It is obvious.
\\
(3 $\Rightarrow$ 4): Suppose ${\mathsf o}(\sigma_{\varphi,\mathfrak{w}}\restriction_{\mathop{\bigoplus}\limits_{\Gamma}F})=\eset(\sigma_{\varphi,\mathfrak{w}}\restriction_{\mathop{\bigoplus}\limits_{\Gamma}F})>0$, thus there exists 
\linebreak
$x\in \mathop{\bigoplus}\limits_{\Gamma}F$
such that $\{\sigma_{\varphi,\mathfrak{w}}^n(x)\}_{n\geq1}(=\{\sigma_{\varphi,\mathfrak{w}}\restriction_{\mathop{\bigoplus}\limits_{\Gamma}F}^n(x)\}_{n\geq1})$. There exist
$\beta_1,\ldots,\beta_p\in\Gamma$ and $r_1,\ldots,r_p\in F$ such that $x=r_1\mathsf{e}_{\beta_1}+\cdots+r_p\mathsf{e}_{\beta_p}$.
Note that $\{\sigma_{\varphi,\mathfrak{w}}^n(x):n\geq1\}=\{r_1\sigma_{\varphi,\mathfrak{w}}^n(\mathsf{e}_{\beta_1})+\cdots+
r_p\sigma_{\varphi,\mathfrak{w}}^n(\mathsf{e}_{\beta_p}):n\geq1\}$ is an infinite subset of 
\[\{r_1\sigma_{\varphi,\mathfrak{w}}^{n_1}(\mathsf{e}_{\beta_1})+\cdots+
r_p\sigma_{\varphi,\mathfrak{w}}^{n_p}(\mathsf{e}_{\beta_p}):n_1,\ldots,n_p\geq1\}\:.\]
So there exists $j\in\{1,\ldots,p\}$ such that
$\{\sigma_{\varphi,\mathfrak{w}}^{n}(\mathsf{e}_{\beta_j}):n\geq1\}$ is an infinite set. Therefore
$\{\sigma_{\varphi,\mathfrak{w}}^{n}(\mathsf{e}_{\beta_j})\}_{n\geq1}$ is a one to one sequence.
\\
(4 $\Rightarrow$ 1) Use Lemma~\ref{kheybar60}.
\end{proof}
 In the following example we show one may choose appropriate $\varphi$ and $\mathfrak{w}$ such that
$\eset(\sigma_{\varphi,\mathfrak{w}}\restriction_{\mathop{\bigoplus}\limits_{\Gamma}F})\neq
\eset(\sigma_{\varphi,\mathfrak{w}})$, hence $\eset(\sigma_{\varphi,\mathfrak{w}}\restriction_{\mathop{\bigoplus}\limits_{\Gamma}F})<
\eset(\sigma_{\varphi,\mathfrak{w}})$.
\begin{exam}
Consider $\eta:\mathbb{N}\to\mathbb{N}$ as in Example~\ref{kheybar40}, and $\mathfrak{u}=(1)_{n\in\mathbb{N}}$, then
$\eset(\sigma_{\eta}\restriction_{\mathop{\bigoplus}\limits_{\mathbb{N}}F})=
\eset(\sigma_{\eta,\mathfrak{v}}\restriction_{\mathop{\bigoplus}\limits_{\mathbb{N}}F})=0<+\infty=
\eset(\sigma_{\eta,\mathfrak{v}})=\eset(\sigma_{\eta})$.
\end{exam}
\section{Contravariant set--theoretical entropy of the finite fibre $\sigma_{\varphi,\mathfrak{w}}\restriction_{\mathop{\bigoplus}\limits_{\Gamma}F}$ whenever
$\mathop{\bigoplus}\limits_{\Gamma}F$ is $\sigma_{\varphi,\mathfrak{w}}-$invariant}\label{maryam50}
\noindent As it has been mentioned in Convention~\ref{kheybar75}, in this section we assume
$\sigma_{\varphi,\mathfrak{w}}(\mathop{\bigoplus}\limits_{\Gamma}F)\subseteq \mathop{\bigoplus}\limits_{\Gamma}F$
(or equivalently, by Lemma~\ref{kheybar30},  $\varphi\restriction_{{\rm supp}(\mathfrak{w})}:{\rm supp}(\mathfrak{w})\to\Gamma$ is of finite fibre).
In this section we show $\sigma_{\varphi,\mathfrak{w}}\restriction_{\mathop{\bigoplus}\limits_{\Gamma}F}:\mathop{\bigoplus}\limits_{\Gamma}F
\to \mathop{\bigoplus}\limits_{\Gamma}F$ is of finite fibre if and only if 
$\sigma_{\varphi,\mathfrak{w}}:F^\Gamma\to F^\Gamma$ is of finite fibre and in the above case
$\ecset(\sigma_{\varphi,\mathfrak{w}}\restriction_{\mathop{\bigoplus}\limits_{\Gamma}F})=+\infty$ if and only if there exists
a $\varphi-$orbit one to one sequence in ${\rm supp}(\mathfrak{w})$.
\begin{proposition}\label{kheybar90}
$\sigma_{\varphi,\mathfrak{w}}\restriction_{\mathop{\bigoplus}\limits_{\Gamma}F}:\mathop{\bigoplus}\limits_{\Gamma}F
\to \mathop{\bigoplus}\limits_{\Gamma}F$ is of finite fibre if and only if 
$\sigma_{\varphi,\mathfrak{w}}:F^\Gamma\to F^\Gamma$ is of finite fibre.
\end{proposition}
\begin{proof}
Using the proof of Lemma~\ref{badr30}
for each $x=(x_\alpha)_{\alpha\in\Gamma}\in \mathop{\bigoplus}\limits_{\Gamma}F$, we have
\[\sigma_{\varphi,\mathfrak w}\restriction_{\mathop{\bigoplus}\limits_{\Gamma}F}^{-1}(\sigma_{\varphi,\mathfrak w}(x))=
 \{(y_\alpha)_{\alpha\in\Gamma}\in\mathop{\bigoplus}\limits_{\Gamma}F
:\forall\alpha\in\varphi(\rm{supp}(\mathfrak{w})),
\SP
    y_\alpha= x_\alpha\}
\]
hence $\sigma_{\varphi,\mathfrak w}\restriction_{\mathop{\bigoplus}\limits_{\Gamma}F}^{-1}(\sigma_{\varphi,\mathfrak w}(x))$ and $\mathop{\bigoplus}\limits_{\Gamma\setminus
\varphi(\rm{supp}(\mathfrak{w}))}F$ are equipotent. Therefore $\sigma_{\varphi,\mathfrak w}\restriction_{\mathop{\bigoplus}\limits_{\Gamma}F}$ is of finite
fibre if and only if $\Gamma\setminus
\varphi(\rm{supp}(\mathfrak{w}))$ is finite. So by Lemma~\ref{badr30}, 
\linebreak
$\sigma_{\varphi,\mathfrak{w}}\restriction_{\mathop{\bigoplus}\limits_{\Gamma}F}:\mathop{\bigoplus}\limits_{\Gamma}F
\to \mathop{\bigoplus}\limits_{\Gamma}F$ is of finite fibre if and only if 
$\sigma_{\varphi,\mathfrak{w}}:F^\Gamma\to F^\Gamma$ is of finite fibre.
\end{proof}
\begin{lemma}\label{kheybar80}
Suppose $\sigma_{\varphi,\mathfrak{w}}\restriction_{\mathop{\bigoplus}\limits_{\Gamma}F}:\mathop{\bigoplus}\limits_{\Gamma}F
\to \mathop{\bigoplus}\limits_{\Gamma}F$ is of finite fibre, then:
\begin{itemize}
\item[1.] ${\rm sc}(\sigma_{\varphi,\mathfrak{w}}\restriction_{\mathop{\bigoplus}\limits_{\Gamma}F})\subseteq{\rm sc}(\sigma_{\varphi,\mathfrak{w}})\cap \mathop{\bigoplus}\limits_{\Gamma}F\subseteq\left\{(x_\alpha)_{\alpha\in\Gamma}\in \mathop{\bigoplus}\limits_{\Gamma}F:\forall\beta\in\Upsilon,
\:\:x_\beta=0\right\}$,
\item[2.] $\sigma_{\varphi\restriction_\Lambda,\mathfrak{w}^\Lambda}(\mathop{\bigoplus}\limits_{\Lambda}F)\subseteq \mathop{\bigoplus}\limits_{\Lambda}F$,
\item[3.] $\sigma_{\varphi\restriction_\Lambda,\mathfrak{w}^\Lambda}\restriction_{\mathop{\bigoplus}\limits_{\Lambda}F}:
	\mathop{\bigoplus}\limits_{\Lambda}F\to \mathop{\bigoplus}\limits_{\Lambda}F$ is of finite fibre,
\item[4.]  $\mathsf{a}(\sigma_{\varphi,\mathfrak{w}}\restriction_{\mathop{\bigoplus}\limits_{\Gamma}F})=\mathsf{a}(\sigma_{\varphi\restriction_\Lambda,\mathfrak{w}^\Lambda}\restriction_{\mathop{\bigoplus}\limits_{\Lambda}F})$.
\end{itemize}
\end{lemma}
\begin{proof}
1) Use Lemma~\ref{badr10}.
\\
2) $\sigma_{\varphi,\mathfrak{w}}(\mathop{\bigoplus}\limits_{\Gamma}F)\subseteq \mathop{\bigoplus}\limits_{\Gamma}F$,
i.e.,  $\varphi\restriction_{{\rm supp}(\mathfrak{w})}:{\rm supp}(\mathfrak{w})\to\Gamma$ is of finite fibre. Since 
\linebreak
$\Lambda\subseteq
{\rm supp}(\mathfrak{w})$, $\varphi\restriction_{\Lambda}:\Lambda\to\Gamma$ is of finite fibre too. Thus
$\varphi\restriction_{\Lambda}:\Lambda\to\varphi(\Lambda)$ is of finite fibre. Therefore 
$\varphi\restriction_{\Lambda}:\Lambda\to\Lambda$ is of finite fibre (use $\varphi(\Lambda)\subseteq\Lambda$).
By ${\rm supp}(\mathfrak{w}^\Lambda)=\Lambda$ and 
Lemma~\ref{kheybar30}, $\sigma_{\varphi\restriction_\Lambda,\mathfrak{w}^\Lambda}(\mathop{\bigoplus}\limits_{\Lambda}F)\subseteq \mathop{\bigoplus}\limits_{\Lambda}F$.
\\
3) By~Proposition~\ref{kheybar90}, $\sigma_{\varphi,\mathfrak{w}}:F^\Gamma\to F^\Gamma$ is of finite fibre. 
By~Corollary~\ref{jadid100}, 
\linebreak
$\sigma_{\varphi\restriction_\Lambda,\mathfrak{w}^\Lambda}:F^\Lambda\to F^\Lambda$
is of finite fibre. By~Proposition~\ref{kheybar90}, $\sigma_{\varphi\restriction_\Lambda,\mathfrak{w}^\Lambda}\restriction_{\mathop{\bigoplus}\limits_{\Lambda}F}:
	\mathop{\bigoplus}\limits_{\Lambda}F\to \mathop{\bigoplus}\limits_{\Lambda}F$ is of finite fibre.
\\
4) Use (1) and a similar proof described in Corollary~\ref{badr25}.
\end{proof}
\begin{lemma}\label{kheybar200}
Suppose $\sigma_{\varphi,\mathfrak{w}}\restriction_{\mathop{\bigoplus}\limits_{\Gamma}F}:\mathop{\bigoplus}\limits_{\Gamma}F
\to \mathop{\bigoplus}\limits_{\Gamma}F$ is of finite fibre, then the following statements are equivalent:

1. $\ecset(\sigma_{\varphi,\mathfrak{w}}\restriction_{\mathop{\bigoplus}\limits_{\Gamma}F})=+\infty$,

2. $\ecset(\sigma_{\varphi,\mathfrak{w}}\restriction_{\mathop{\bigoplus}\limits_{\Gamma}F})>0$,

3. there exists a one to one $\varphi-$orbit sequence in ${\rm supp}(\mathfrak{w})$.
\\
$\:$ \\
So:
{\small \[\ecset(\sigma_{\varphi,\mathfrak{w}}\restriction_{\mathop{\bigoplus}\limits_{\Gamma}F})=\left\{\begin{array}{lc}
+\infty\:, &  if \: there \: exists \: a \: one \: to \: one \: \varphi-orbit \: sequence \: in \: {\rm supp}(\mathfrak{w}) \:, \\
0\:, & otherwise\:.
\end{array}\right.\]}
$\:$ \\ $\:$
\end{lemma}
\begin{proof}
(1 $\Rightarrow$ 2): It is obvious.
\\
(2 $\Rightarrow$ 3): Suppose
$\mathsf{a}(\sigma_{\varphi,\mathfrak{w}}\restriction_{\mathop{\bigoplus}\limits_{\Gamma}F})=
\ecset(\sigma_{\varphi,\mathfrak{w}}\restriction_{\mathop{\bigoplus}\limits_{\Gamma}F})>0$.
By Lemma~\ref{kheybar80}(4), 
$\mathsf{a}(\sigma_{\varphi\restriction_\Lambda,\mathfrak{w}^\Lambda}\restriction_{\mathop{\bigoplus}\limits_{\Lambda}F})>0$,
hence there exists a 
one to one $\sigma_{\varphi\restriction_\Lambda,\mathfrak{w}^\Lambda}-$anti--orbit sequence 
$\{z_n\}_{n\geq1}$ in $\mathop{\bigoplus}\limits_{\Lambda}F$.
We may suppose $z_n\neq(0)_{\alpha\in\Lambda}$
thus 
${\rm supp}(z_n)\neq\varnothing$ (for all $n\geq1$). 
For $n\geq1$ let $z_n=(z^n_\alpha)_{\alpha\in\Lambda}=\mathop{\Sigma}\limits_{\beta\in{\rm supp}(z_n)}z^n_\beta\mathsf{e}_\beta^\Lambda$, 
then (using a similar method described in the proof of Lemma~\ref{kheybar60}~(1))
\[
\begin{array}{rcl}
z_n & = & \sigma_{\varphi\restriction_\Lambda,\mathfrak{w}^\Lambda}(z_{n+1}) =
	\sigma_{\varphi\restriction_\Lambda,\mathfrak{w}^\Lambda}(\mathop{\Sigma}\limits_{\beta\in{\rm supp}
		(z_{n+1})}z^{n+1}_\beta\mathsf{e}_\beta^\Lambda) \\ && \\
& = & \mathop{\Sigma}\limits_{\beta\in{\rm supp}
		(z_{n+1})}z^{n+1}_\beta\sigma_{\varphi\restriction_\Lambda,		
		\mathfrak{w}^\Lambda}(\mathsf{e}_\beta^\Lambda) \\ && \\
& = & \mathop{\Sigma}\limits_{\beta\in{\rm supp}(z_{n+1})}
	\bigg(\mathop{\Sigma}\limits_{\alpha\in\varphi\restriction_\Lambda^{-1}(\beta)\cap{\rm supp}
	(\mathfrak{w}^\Lambda)}z^{n+1}_\beta\mathfrak{w}_\alpha\mathsf{e}_\alpha\bigg) \\ && \\
& = & \mathop{\Sigma}\limits_{\beta\in{\rm supp}(z_{n+1})}
	\bigg(\mathop{\Sigma}\limits_{\alpha\in\varphi\restriction_\Lambda^{-1}(\beta)}z^{n+1}_\beta\mathfrak{w}_\alpha\mathsf{e}_\alpha\bigg)\:, \\ &&
\end{array}\]
for all $\beta\in{\rm supp}(z_{n+1})$, $z^{n+1}_\beta\neq0$, thus 
\[{\rm supp}(z_{n})=\varphi\restriction_\Lambda^{-1}({\rm supp}(z_{n+1}))\:.\]
Inductively we show:
\begin{equation}\label{kheybar100}
\forall n\geq1 \:\: {\rm supp}(z_n)\subseteq \{\varphi^k(\alpha):\alpha\in{\rm supp}(z_1)\cup(\Lambda\setminus\varphi(\Lambda)),k\geq0\}\:.
\end{equation}
For this aim let $H:=\{\varphi^k(\alpha):\alpha\in{\rm supp}(z_1)\cup(\Lambda\setminus\varphi(\Lambda)),k\geq0\}$ and use the following steps:
\begin{itemize}
\item Obviously ${\rm supp}(z_1)\subseteq H$.
\item[]
\item Consider $m\geq1$ such that ${\rm supp}(z_m)\subseteq \{\varphi^k(\alpha):\alpha\in{\rm supp}(z_1)\cup(\Lambda\setminus\varphi(\Lambda)),k\geq0\}$. If $\beta\in {\rm supp}(z_{m+1})$, then one of the following 
conditions occur:
	\begin{itemize}
	\item Case 1: $\beta\notin\varphi(\Lambda)$. In this case $\beta\in \Lambda\setminus\varphi(\Lambda)\subseteq H$.
	\item Case 2: $\beta\in\varphi(\Lambda)$. In this case choose $\theta\in\Lambda$ such that $\varphi(\theta)=\beta$ thus 
\linebreak
	$\theta\in \varphi\restriction_\Lambda^{-1}({\rm supp}(z_{m+1}))={\rm supp}(z_m)\subseteq H$. Therefore
	$\beta=\varphi(\theta)\subseteq\varphi(H)$.
\item[]
	\end{itemize}
Using the above cases $\beta\in H\cup \varphi(H)\subseteq H$. Hence ${\rm supp}(z_{m+1})\subseteq H$.
\end{itemize}
Therefore~\ref{kheybar100} is valid. By~\ref{kheybar80}(3) the set $\Lambda\setminus\varphi(\Lambda)$ is finite.
So ${\rm supp}(z_1)\cup(\Lambda\setminus\varphi(\Lambda))=\{\theta_1,\ldots,\theta_p\}$ is finite. Suppose
\[V_i:=\{r_1\mathsf{e}_{\varphi^{n_1}(\theta_i)}^\Lambda+\cdots+r_k\mathsf{e}_{\varphi^{n_k}(\theta_i)}^\Lambda:
k\geq1, r_1,\ldots,r_k\in F, n_1,\ldots,n_k\geq0\}\]
is the linear subspace generated by $\{\mathsf{e}_{\varphi^{n}(\theta_i)}:n\geq0\}$ ($1\leq i\leq p$).
Hence 
\vspace*{3mm}
{\small\begin{eqnarray*}
\{z_n&:&n\geq1\} = \left\{\mathop{\Sigma}\limits_{\beta\in{\rm supp}(z_n)}z^n_\beta\mathsf{e}_\beta^\Lambda:n\geq1\right\} \\
	& \subseteq & \{r_1\mathsf{e}_{\beta_1}^\Lambda+\cdots+r_m\mathsf{e}_{\beta_m}^\Lambda:
		m\geq1,\beta_1,\ldots,\beta_m\in\bigcup\{{\rm supp}(z_n):n\geq1\}, r_1,\ldots,r_m\in F\} \\
	& \mathop{\subseteq}\limits^{\ref{kheybar100}} &  \{r_1\mathsf{e}_{\beta_1}^\Lambda+\cdots+r_m\mathsf{e}_{\beta_m}^\Lambda:
		m\geq1,\beta_1,\ldots,\beta_m\in H, r_1,\ldots,r_m\in F\} \\
	& = & V_1+\cdots +V_p\:.
\end{eqnarray*}}
$\{z_n:n\geq1\}$ is an infinite subset of $V_1+\cdots +V_p$. Therefore there exists $j\in\{1,\ldots,p\}$ such that $V_j$ is infinite,
thus
$\{\mathsf{e}_{\varphi^{n}(\theta_j)}:n\geq0\}$ is infinite, i.e. $\{\varphi^{n}(\theta_j):n\geq0\}$ is infinite
and $\{\varphi^{n}(\theta_j)\}_{n\geq1}$ is a one to one $\varphi\restriction_\Lambda-$orbit sequence.
In particular $\{\varphi^{n}(\theta_j)\}_{n\geq1}$ is a one to one $\varphi-$orbit sequence and
for all $n\geq1$, $\varphi^{n}(\theta_j)\in\Lambda\subseteq{\rm supp}(\mathfrak{w})$. 
The sequence $\{\varphi^{n}(\theta_j)\}_{n\geq1}$ satisfies (3).
\\
(3 $\Rightarrow$ 1): Suppose $\{\beta_n\}_{n\geq1}$ is a one to one sequence in ${\rm supp}(\mathfrak{w})$.
For $r\in F$ consider $r^*$ as \ref{kheybar300} in the proof of Theorem~\ref{badr40}. Also, consider 
$\eta:\mathop{\mathbb{N}\to\mathbb{N}\:\:}\limits_{n\mapsto n+1}$ and
surjection
\linebreak
$k:\mathop{\bigoplus}\limits_{\Gamma}F\to\mathop{\bigoplus}\limits_{\mathbb{N}}\{0,1\}$ with
$k((x_\alpha)_{\alpha\in\Gamma})=(x_{\beta_n}^*)_{n\in\mathbb{N}}$ ($(x_\alpha)_{\alpha\in\Gamma}\in \mathop{\bigoplus}\limits_{\Gamma}F$).
\\
For each
$(x_\alpha)_{\alpha\in\Gamma}\in \mathop{\bigoplus}\limits_{\Gamma}F$, we have:
\begin{eqnarray*}
k(\sigma_{\varphi,\mathfrak{w}}((x_\alpha)_{\alpha\in\Gamma})) & = &
    k((\mathfrak{w}_\alpha x_{\varphi(\alpha)})_{\alpha\in\Gamma}) =
    ((\mathfrak{w}_{\beta_n} x_{\varphi(\beta_n)})^*)_{n\in\mathbb{N}} \\
& = & (\mathfrak{w}_{\beta_n}^* x_{\beta_{n+1}}^*)_{n\in\mathbb{N}}\mathop{=}\limits^{(\beta_k\in{\rm supp}(\mathfrak{w}))}
    (x_{\beta_{n+1}}^*)_{n\in\mathbb{N}} \\
& = & \sigma_\eta((x_{\beta_n}^*)_{n\in\mathbb{N}})=\sigma_\eta(k((x_\alpha)_{\alpha\in\Gamma}))
\end{eqnarray*}
Hence $k\circ(\sigma_{\varphi,\mathfrak{w}}\restriction_{\mathop{\bigoplus}\limits_{\Gamma}F})=
(\sigma_\eta\restriction_{\mathop{\bigoplus}\limits_{\mathbb{N}}\{0,1\}})\circ k$ and the following diagram commutes:
\[\xymatrix{\mathop{\bigoplus}\limits_{\Gamma}F\ar[r]^{\sigma_{\varphi,\mathfrak{w}}\restriction_{\mathop{\bigoplus}\limits_{\Gamma}F}}\ar[d]_k & \mathop{\bigoplus}\limits_{\Gamma}F\ar[d]^k \\
\mathop{\bigoplus}\limits_{\mathbb{N}}\{0,1\}\ar[r]^{\sigma_\eta\restriction_{\mathop{\bigoplus}\limits_{\mathbb{N}}\{0,1\}}} & \mathop{\bigoplus}\limits_{\mathbb{N}}\{0,1\} }\]
By ~\cite[Lemma 3.2.22 (b)]{anna} we have:
\begin{equation}\label{kheybar400}
\ecset(\sigma_\eta\restriction_{\mathop{\bigoplus}\limits_{\mathbb{N}}\{0,1\}})\leq\ecset(\sigma_{\varphi,\mathfrak{w}}\restriction_{\mathop{\bigoplus}\limits_{\Gamma}F})\:.
\end{equation}
Let
\[a_n^m=(\underbrace{0,\cdots,0}_{n\: times},\underbrace{1,\cdots,1}_{m\: times},0,0,0,\cdots)\SP(n,m\geq1)\:.\]
Then $\{a_n^1\}_{n\geq1}, \{a_n^2\}_{n\geq1},\{a_n^3\}_{n\geq1},\ldots$ are pairwise disjoint one to one
$\sigma_\eta\restriction_{\mathop{\bigoplus}\limits_{\mathbb{N}}\{0,1\}}-$anti--orbit sequences, thus
$\ecset(\sigma_\eta\restriction_{\mathop{\bigoplus}\limits_{\mathbb{N}}\{0,1\}})=\mathsf{a}(
\sigma_\eta\restriction_{\mathop{\bigoplus}\limits_{\mathbb{N}}\{0,1\}})=+\infty$ which completes the proof by~\ref{kheybar400}.
\end{proof}
\newpage
\section*{Acknowledgement} 
\noindent The authors wish to express their thanks to the anonymous referee for his/her useful guides.
Also with thanks to the research division of Farhangian University for the grant which supported this research.

\vspace{2mm}
\[\underline{\SP\SP\SP\SP\SP\SP\SP\SP\SP\SP\SP\SP\SP\SP\SP\SP}\]
\noindent {\bf Fatemah Ayatollah Zadeh Shirazi}, Faculty
of Mathematics, Statistics and Computer Science, College of
Science, University of Tehran, Enghelab Ave., Tehran, Iran
(f.a.z.shirazi@ut.ac.ir)
\\
{\bf Arezoo Hosseini},
Department of Mathematics Education, Farhangian University,
\linebreak
P.~O.~Box 14665--889, Tehran, Iran
(a.hosseini@cfu.ac.ir)
\\
{\bf Lida Mousavi}, Department of Mathematics, Yadegar-e-Imam Khomeini (RAH), Shahre
Rey Branch, Islamic Azad University, Tehran, Iran (mousavi.lida@gmail.com)
\\
{\bf Reza Rezavand}, School of Mathematics, Statistics
and Computer Science, College of Science, University of Tehran,
Enghelab Ave., Tehran, Iran (rezavand@ut.ac.ir)


\begin{thebibliography}{99}
\vspace{1mm}
\bibitem{adler1} R. Adler, A. Konheim, M. McAndrew, \textit{Topological entropy}, Trans. AMS 114 (1965) 
\linebreak
309--319. 
\vspace{1mm}
\bibitem{adler} R. L. Adler, B. Marcus, \textit{Topological entropy and equivalence of dynamical systems} Mem. Amer. Math. Soc. 20 , no. 219, 1979.
\vspace{1mm}
\bibitem{qm} M. Akhavin, F. Ayatollah Zadeh Shirazi, D. Dikranjan, A. Giordano Bruno,
A.,Hosseini, \textit{Algebraic entropy of shift endomorphisms on abelian groups},
Quaestiones Mathematicae, 32/4 (2009), 529--550.
\vspace{1mm}
\bibitem{md} F. Ayatollah Zadeh Shirazi, D. Dikranjan, \textit{Set theoretical entropy: A tool to compute topological entropy},  Proceedings ICTA 2011, Islamabad, Pakistan, July 4-10, 2011 (Cambridge Scientific Publishers), 2012, 11--32.
\vspace{1mm}
\bibitem{lp} F. Ayatollah Zadeh Shirazi, F. Ebrahimifar, R. Rezavand, \textit{Weighted generalized shift operators on $\ell^p$ spaces}, Rendiconti del Circolo Matematico di Palermo Series 2, 71/1 (2022), 
\linebreak
1--12.
\vspace{1mm}
\bibitem{note} F. Ayatollah Zadeh Shirazi, N. Karami Kabir, F. Heidari Ardi, 	\textit{A Note on shift theory}, Mathematica Pannonica, Proceedings of ITES-2007, 19/2 (2008), 187--195.
\vspace{1mm}
\bibitem{cmuc} F. Ayatollah Zadeh Shirazi, S. Karimzadeh Dolatabad, S. Shamloo,  \textit{Interaction
between cellularity of alexandroff spaces and entropy of generalized shift maps}, Commentationes Mathematicae Universitatis Carolinae, 27/3 (2016), 397--410.
\vspace{1mm}
\bibitem{anna} D. Dikranjan, A, Giordano Bruno, \textit{Topological entropy and algebraic entropy for group endomorphisms}, Proceedings ICTA 2011, Islamabad, Pakistan, July 4-10, 2011 (Cambridge Scientific Publishers), 2012, 133--214.
\vspace{1mm}
\bibitem{adjoint} D. Dikranjan, A. Giordano Bruno, L. Salce, \textit{Adjoint algebraic entropy}, J. Algebra, 324, no. 3 (2010), 442--463.
\vspace{1mm}
\bibitem{algd}  D. Dikranjan, B. Goldsmith, L. Salce, P. Zanardo, \textit{Algebraic entropy for abelian groups}, Trans. Amer. Math. Soc. 361 (2009) 3401–3434.
\vspace{1mm}
\bibitem{string} D. Dikranjan, A.  Giordano Bruno, S. Virili,  \textit{Strings of group endomorphisms}, J. Algebra Appl, 9/6 (2010), 933--958.
\vspace{1mm}
\bibitem{gior}  A. Giordano Bruno, \textit{Algebraic entropy of shift endomorphisms on products}, Comm. Algebra 38 (11) (2010) 4155--4174.
\vspace{1mm}
\bibitem{kolomogrof}  A. N. Kolmogorov, \textit{New metric invariants of transitive dynamical systems and automorphisms of Lebesgue spaces}, Doklady Akad. Nauk. SSSR 119 (1958) 861–864. 
\vspace{1mm}
\bibitem{nili} Z. Nili Ahmadabadi, F. Ayatollah Zadeh Shirazi, \textit{Set--theoretical entropies of generalized shifts}, Boletin de Matematicas,  24/2 (2017), 169--183 
\vspace{1mm}
\bibitem{orn1} D. Ornstein, \textit{Bernoulli shifts with the same entropy are isomorphic}, Advances in
Mathematics, 4 (1970), 337--352.
\vspace{1mm}
\bibitem{orn2} D. Ornstein, \textit{Two bernoulli shifts with infinite entropy are isomorphic}, Advances
in Mathematics, 5 (1970-1971), 339--348.
\vspace{1mm}
\bibitem{neww} B. Seward, \textit{Bernoulli shifts with bases of equal entropy are isomorphic}, Journal
of Modern Dynamics, 18 (2022), 345--362.
\vspace{1mm}
\bibitem{bern} P. Shields, \textit{The Theory of Bernoulli Shifts} Chicago Lectures in Mathematics,
The University of Chicago Press, Chicago 1973.
\vspace{1mm}
\bibitem{sinai} Y. G. Sinai, \textit{On the concept of entropy of a dynamical system}, Doklady Akad. Nauk. SSSR 124 (1959) 786–781. 
\vspace{1mm}
\bibitem{walters} P. Walters, \textit{An introduction to ergodic theory}, Graduate Texts in Mathematics, 79. Springer--Verlag, 1982. \vspace{1mm}
\bibitem{weiss}  M. D. Weiss, \textit{Algebraic and other entropies of group endomorphisms}, Math. Systems Theory 8 (3) (1974/75) 243– 248. 
\end{thebibliography}
\end{document}